\documentclass[11pt]{article}

\usepackage{amsmath,amssymb, amsthm}
\usepackage{mathtools}
\usepackage{bm}
\usepackage{booktabs} 
\usepackage{subcaption}
\usepackage{ascmac}
\usepackage{graphicx} 
\usepackage[colorlinks=true,linkcolor=blue,citecolor=blue,urlcolor=blue]{hyperref}
\usepackage{tikz}
\usetikzlibrary{intersections, calc, arrows, math}
\usepackage{multirow}
\usepackage[normalem]{ulem}
\usepackage{url}

\newcommand{\mat}[1]{\begin{bmatrix} #1 \end{bmatrix}}
\newcommand{\Rn}{\mathbb{R}^n}
\newcommand{\R}{\mathbb{R}}
\newcommand{\T}{\top\hspace{-1pt}}

\newcommand{\argmin}{\mathop{\mathrm{arg~min}}\limits}
\newcommand{\bigmid}{\mathrel{}\middle|\mathrel{}}
\newcommand{\prob}[1]{\mathbb{P} \left( {#1} \right)}
\newcommand{\diag}{\operatorname{diag}}

\definecolor{myred}{RGB}{255,50,50}         
\definecolor{myblack}{RGB}{0,0,0}           
\definecolor{myblue}{RGB}{0,0,210}

\def\qed{\hfill $\Box$}
\newtheorem{thm}{Theorem}[section]
\newtheorem{lemma}[thm]{Lemma}

\newtheorem{dfn}[thm]{Definition}
\newtheorem{remark}{Remark}[section]
\newtheorem{assum}{Assumption}[section]
\newtheorem{example}{Example}[section]
\newtheorem{prop}[thm]{Proposition}
\newtheorem{cor}[thm]{Corollary}

\numberwithin{equation}{section}

\title{An uncertainty model for positive-valued parameters 
      with application to robust optimization%
  \thanks{This work was supported by JST SPRING (JPMJSP2110), the Grant-in-Aid for Scientific Research (C) (JP25K15002) and Grant-in-Aid for Scientific Research (B) (JP25K03082) from Japan Society for the Promotion of Science, and partially supported by the joint project of Kyoto University and Toyota Motor Corporation, titled ``Advanced Mathematical Science for Mobility Society''.}
}
\author{
  Tatsuya Tanaka$^{1}$, Huimin Li$^{1}$, Shota Yamanaka$^{2}$,\\[5pt]
  Ellen H. Fukuda$^{1}$, Nobuo Yamashita$^{1}$\\[0.75ex]
  \small $^{1}$Graduate School of Informatics, Kyoto University, Kyoto 606--8501, Japan\\
  \small $^{2}$Frontier Research Center, Toyota Motor Corporation, Aichi 471--8572, Japan
}

\begin{document}

\maketitle

\begin{abstract}
\noindent
Many practical optimization problems involve uncertain parameters that are strictly positive. However, the most common uncertainty sets used in robust optimization are the box and the ellipsoidal sets, which may include non-positive values when the level of uncertainty is large. This can lead to overly conservative solutions or make the corresponding robust counterpart infeasible. To overcome this, in this paper, we propose a new uncertainty-set model that not only preserves positivity but is also computationally tractable. The proposed set uses a particular convex function that measures the variation of uncertain parameters from their nominal values. We can also write the dual reformulation of the associated robust problem. For the theoretical results, we show several properties of the proposed model, including analytical bounds that guide the choice of the uncertainty level, as well as a probabilistic guarantee result. To check the validity of our proposal, we consider photovoltaic--battery operation planning problems and support vector machines in the numerical experiments. For these problems, standard uncertainty models may lead to infeasibility of the robust counterpart, while the proposed uncertainty set gives a tractable dual reformulation.\\

\noindent \textbf{Keywords:} Robust optimization, uncertainty set, positive-valued parameters, Fenchel duality, probabilistic guarantee.\\

\noindent \textbf{Mathematics subject classification:} 90C15, 90C46, 90C90
\end{abstract}

\section{Introduction} \label{sect:intro}

Robust optimization is one of the modeling frameworks in mathematical optimization that deals with uncertainties. 
When the uncertainty lies in the data, it assumes that the parameters belong to an uncertainty set and optimizes the decision variables under the worst-case scenario within that set. This reformulated problem is called the robust counterpart.
In most real-world applications of mathematical optimization, we do not know the exact values of the parameters or data; that is, their uncertainty is unavoidable.
For this reason, robust optimization attracts substantial attention as a methodology for obtaining reliable solutions even under such uncertainty.

The research related to robust optimization is vast. The paper~\cite{Ben00} reported cases taken from the well-known NETLIB library, where optimal solutions to the original nominal optimization problems violate their constraints with only a slight perturbation of the parameters.
This result demonstrates the critical importance of dealing with uncertainties. In terms of uncertainty modeling, the box uncertainty set, which is one of the most basic models, was proposed in~\cite{Soy73}.
It sets lower and upper bounds for the uncertain parameter vectors, and if the nominal problem is linear, then we can reformulate its robust counterpart as a linear program.
Another well-known uncertainty set is the ellipsoidal one,
proposed in~\cite{Ben98}, which expresses the uncertainty with an ellipsoid centered at a nominal value.
Because it can account for the correlation between each element of the uncertain parameters, it tends to yield less conservative robust optimal solutions compared to the box model.
In this case, if the nominal problem is linear, its robust counterpart can be written as a second-order cone optimization problem.
Note also that there is a fundamental trade-off between robustness and optimality: the higher the level of uncertainty we assume, the less optimal the robust solution becomes for the nominal problem.
An analysis of this trade-off, the ``price of robustness,'' was presented in~\cite{Ber04}, where the authors also proposed a linear uncertainty model which assumes that only some elements of the uncertain parameter vector can change from a nominal value at the same time.

Besides uncertainty modeling, it is also important to examine the tractability of robust counterparts.
The work~\cite{Ben15} proposed reformulating robust counterparts into computationally tractable problems using Fenchel duality. It showed that if a nominal constraint is concave with respect to an uncertain parameter vector and a particular class of uncertainty set is used, the problem can be reformulated into a tractable form.
This transformed dual form is convex with respect to the decision variables if the original constraint is also convex for them.
In addition, applications of robust optimization have been studied extensively, including truss topology design~\cite{Ben97}, portfolio selection~\cite{Fab07}, antenna design~\cite{Ben02}, optimization of energy systems~\cite{Ber12,Jia11}, water infrastructure planning~\cite{Beh17}, and classification problems in machine learning~\cite{Tak13,Xu09}.

On the other hand, there is a need for more uncertainty models suited to specific classes of parameter uncertainty. In particular, many practical optimization problems involve uncertain parameters that are strictly positive. For example, production costs in transport programming, stock prices in financial engineering, power demand in the optimization of energy systems or their operation, and bar volumes in truss topology design are all strictly positive-valued parameters. If we use uncertainty sets defined in the $m$-dimensional real space $\R^m$ without enforcing positivity, they may include zero or negative values at high uncertainty levels, which can lead to overly conservative solutions or even infeasibility of the robust counterpart. Therefore, we need an uncertainty model that can handle positive-valued parameters while still admitting a tractable reformulation.

Motivated by this, we propose a new uncertainty set model that is both tractable and capable of handling positive-valued parameters.
This model utilizes a convex function $t \mapsto t - \ln t - 1$ to measure the variation of the uncertain parameters from their nominal values.
We derive a tractable dual expression of the robust counterpart with our proposed uncertainty model by adapting and using the existing Fenchel duality results.

We establish several properties of robust optimization problems that utilize the proposed model as an uncertainty set.
We focus on the worst-case assumed value of the uncertain parameter and the robust optimal value, deriving their boundedness for arbitrary levels of assumed uncertainty.
In addition, we define their asymptotic properties as the uncertainty level increases.
These properties demonstrate that our model successfully maintains the ``positivity'' of the uncertain parameters, and they provide a theoretical basis for determining the size of the uncertainty set prior to solving the robust counterpart.
Furthermore, we derive a probabilistic guarantee under the assumption that each element of the uncertain parameter vector independently follows a lognormal distribution, which is one of the most common positive-valued probability distributions.
This guarantee evaluates the probability that the robust feasible solutions remain feasible for the original problem with respect to realized parameter scenarios.
We also conduct numerical experiments applying our model to two application problems: power system operational planning, and classification with support vector machines.

This paper is organized as follows. In Section~\ref{sect:pre}, we present the fundamental notions of robust optimization and tractability results. We also discuss the specific issues that occur when modeling positive-valued uncertainty parameters using existing models.
In Section~\ref{sect:log}, we propose a new uncertainty model with positive-valued uncertain parameters, and we derive its tractable dual expression by adapting and using the results from~\cite{Ben15}.
In Section~\ref{sect:prop}, we show several properties that result from applying the proposed model to construct the robust counterpart. These properties indicate that our model is highly suitable for assuming the uncertainty of positive-valued parameters. We also introduce a methodology for determining the values of the model's parameters and derive the probabilistic guarantee. In Section~\ref{sect:num}, we show some numerical experiments applying the proposed model.
Finally, in Section \ref{sect:conc}, we present some conclusions.


\section{Preliminaries} \label{sect:pre}

In this section, we present the notation used throughout this paper.
Next, we introduce fundamental notions of robust optimization 
and tractability results.
We also refer to the importance of modeling positive-valued uncertainty parameters
and the problems which occur when using existing models.

\subsection{Basic definitions and notation}

Throughout this work we use the following notation.
We denote the entries of a vector $x \in \Rn$ by $x_i \in \R$ for $i = 1, \dots, n$, 
while $x^1, \dots, x^N \in \R^n$ is a sequence of vectors. 
Given any matrix $M \in \R^{m \times n}$, its entries are denoted by $M_{i,j} \in \R$ for $i = 1, \dots, m$ and $j = 1, \dots, n$.
The sets $\R_+ ^n$ and $\R^n_{++}$ represent the nonnegative and positive orthants of $\R^n$, respectively, that is, $\R^n_{+} \coloneqq \left\{ x \in \R^n \bigmid x_i \geq 0, \ i = 1, \dots, n \right\}$ and
$\R^n_{++} \coloneqq \left\{ x \in \R^n \bigmid x_i > 0, \ i = 1, \dots, n \right\}$.
For a vector $x \in \Rn$, $\diag (x)$ denotes
the diagonal matrix in $\R^{n \times n}$ with diagonal entries $x_1, \dots, x_n$. We also use~$e$ for the all-ones vector, i.e., $e \coloneqq (1, \dots, 1)^\T$, where its dimension is determined by context, and $\T$ means transpose.
For a vector or a matrix, $\| \cdot \|$ denotes {an} arbitrary norm defined {on} $\R^n$ or $\R^{m \times n}$, respectively.  In particular, $\| x \|_2$ for a vector $x \in \Rn$ is the Euclidean norm and $\| A \|_F$ for a matrix $A \in \R^{m \times n}$ is the Frobenius norm.
Moreover, for a vector $x \in \Rn$, $\|x\|_*$ denotes the dual norm of the corresponding norm~$\|x\|$.

For a function $f \colon \Rn \rightarrow \R$, 
the set $\textrm{dom} (f) \subseteq \Rn$ denotes the domain of~$f$.
A differentiable function~$f$ is said to be increasing if its partial derivatives satisfy $\partial f(x)/ \partial x_i \geq 0$ for all $i = 1, \dots, n$ and $x \in \textrm{dom} (f)$. 
Similarly, $f$ is said to be decreasing
if $\partial f(x) / \partial x_i \leq 0$ for all $i = 1, \dots, n$ and $x \in \textrm{dom} (f)$.
Besides, the set $\partial f(x)$ denotes the subdifferential of~$f$ at~$x$, and when $f$ is differentiable at~$x$, we use $\nabla f(x)$ for its gradient.
In addition, for a convex function $f \colon \Rn \rightarrow \R$, 
$f^*(\cdot)$ denotes the conjugate function 
of $f$ defined by 
$f^*(v) \coloneqq \sup_{x \in \textrm{dom} (f)} \left\{ v^{\T} x - f(x) \right\}$.
On the other hand, for a concave function $g \colon \Rn \to \R$, $g_*(\cdot)$ denotes 
the conjugate function of $g$ defined as 
$g_*(v) \coloneqq \inf_{x \in \textrm{dom} ({g})} \left\{ v^{\T} x - g(x) \right\}$.
{For a function $h \colon \R^m \times \R^n \to \R$, we denote the partial concave conjugate with respect to the first variable as $h_*(v,x) := \inf_{a \in \R^m} \{ a^\T v - h(a,x) \}$.}
Also, $\delta^*(\cdot \mid S)$ {denotes} the support function for a convex set $S \subseteq \Rn$ 
defined as the conjugate function of the indicator function 
$\delta(\cdot \mid S)$ of $S$. 
Note that from the definition of the indicator function{,}
the support function is written as 
$\delta^*(v \mid S) \coloneqq \sup_{x \in S} v^{\T} x$.
Finally, for a set $S$, we denote its relative interior by $\textrm{ri } S$.

\subsection{Robust optimization}
We consider the following optimization problem:
\begin{equation}
    \label{prob:P_a}
    \tag{$\text{P}_{a}$}
    \begin{aligned}
        &{\displaystyle \min_x} \quad f_0(a, x) \\
        & \ {\rm s.t.} \hspace{5mm} f_i(a,x)\leq 0, \quad
            i = 1, \dots, N,
    \end{aligned} 
\end{equation}
where $x\in\mathbb{R}^n$ is the decision variable, $a\in\mathbb{R}^m$ is a parameter vector,
and $f_i\colon\mathbb{R}^m\times\mathbb{R}^n\to\mathbb{R}$ for $i=0,1,\dots,N$.
To solve the problem~\eqref{prob:P_a},
we can fix the value of parameter $a$ to a specific {nominal value}.
However, it is often the case that
we do not know the value of $a$ {precisely}.
In situations where
the solution of~\eqref{prob:P_a} must satisfy the constraints 
for all possible {realizations} 
or the objective function value should not be much larger than expected  
for every realized {value} of~$a$,
we have to {explicitly account for} the uncertainty of {the} parameter $a$.

In robust optimization, 
we assume the possible values of uncertain parameter $a$
with a set $\mathcal{U} \subseteq \R^m$ called {an} uncertainty set.
{Instead} of \eqref{prob:P_a}, we solve the problem below.
The goal of this problem is to optimize the objective function {with respect} to the {worst-case value} of $a$ in $\mathcal{U}$.  

\begin{dfn}
    We call the following problem a robust counterpart of the problem \eqref{prob:P_a}:
    \begin{equation}
        \label{prob:RC_P}
        \tag{$\text{RC}_{\text{P}}$}
        \begin{aligned}
            &{\displaystyle \min_x} \quad 
                \max_{a \in \mathcal{U}} \ {f_0}(a, x) \\
            & \ {\rm s.t.} \hspace{5mm} f_i (a, x) \le 0 \quad
                \forall {a} \in \mathcal{U}, \
                i = 1, \dots, N.
        \end{aligned}
    \end{equation}    
    Moreover, feasible solutions $x \in \Rn$ to the robust counterpart \eqref{prob:RC_P} 
    are called robust feasible solutions of the nominal problem \eqref{prob:P_a}.
    Similarly, the optimal solutions $x^* \in \Rn$ and the optimal value ${f_0}^*$ 
    of the robust counterpart \eqref{prob:RC_P}
    are called robust optimal solutions and the robust optimal value of the nominal problem \eqref{prob:P_a}, respectively. 
\end{dfn}
\noindent In contrast to the robust counterpart \eqref{prob:RC_P}, 
we {refer to} the original problem~\eqref{prob:P_a} {as} the nominal problem.

The robust counterpart is in general a semi-infinite problem,
which cannot be numerically solved efficiently even if it is convex.
Whether we can convert it into a tractable form and solve it efficiently
depends on the structure of the nominal problem~\eqref{prob:P_a}
and the choice of $\mathcal{U}$.
For instance, if {each} $f_i$, {$i = 1, \dots, N$} is affine with respect to  $a$,
and $\mathcal{U}$ is ${a}$ box~\cite{Soy73} or ellipsoidal~\cite{Ben98} uncertainty set, 
we can {explicitly} express the optimal solution $a^*$ of the maximization problem 
{appearing} in the constraints of the robust counterpart \eqref{prob:RC_P}.
{In this case,} the robust constraints {can be written} as linear or second-order cone constraints, respectively.
It is well known that we can efficiently solve these classes of problems 
with existing optimization algorithms and solvers.
However, in most cases, we cannot solve the maximization in the robust constraints explicitly. 
On the other hand, for robust constraints of the form
$
\max_{a \in U} f(a,x) \le 0,
$
if \(f(\cdot,x)\) is concave for all \(x\), then the constraint can be reformulated using Fenchel duality, as explained below.

\subsection{Tractable dual expression}
Let us now focus on a single robust constraint of the form
\begin{equation} \label{constraint:robust_U}
    \max_{a \in \mathcal{U}} f(a, x) \leq 0,
\end{equation}
where $f\colon\mathbb{R}^m\times\mathbb{R}^n\to\mathbb{R}$ and $f(\cdot,x)$ is concave for all
$x\in\mathbb{R}^n$. Then, we can reformulate this constraint into a tractable form using
Fenchel duality. The work~\cite{Ben15} showed that the robust constraint above can be reformulated
without the maximization term.

\begin{thm}[{\cite[Theorem 2]{Ben15}}] \label{Ben_dual}
    Suppose that the function $f \colon \R^m \times \R^n \to \R$
    {is such that $f(\cdot,x)$ is concave for all $x \in \R^n$.}
    We also assume that the nominal value $a^0$ satisfies
    $a^0 \in \text{\rm ri dom}(f(\cdot, x))$
    for all $x$.
    Let $\mathcal{U}$ be an uncertainty set defined as  
    \begin{equation}
        \mathcal{U} \coloneqq \left\{ a \in \mathbb{R}^m \bigmid a = a^0 + B \zeta, \ \zeta \in Z \right\},
         \label{Ben_uncertainty_set}
    \end{equation}
    where {$B \in \R^{m \times l}$ and} $Z \subset \R^l$ is a nonempty compact convex set {with}
    $0 \in {\rm ri} \ Z$.
    Then $x \in \Rn$ satisfies the robust constraint (\ref{constraint:robust_U})
    if and only if {there exists} a vector $v \in \R^m$
    satisfying the inequality  
    \begin{equation} \label{constraint:Ben_dual}
          (a^0 )^{\T} v + \delta^* (B^{\T} v \mid Z ) 
          - f_*(v, x) \leq 0. 
    \end{equation}
    \qed
\end{thm}

Note that~\eqref{constraint:Ben_dual} holds under the assumption that $f(\cdot,x)$ is concave for all~$x$. However, if in addition we assume that $f(a,\cdot)$ is convex for all~$a$, then~\eqref{constraint:Ben_dual}
is a convex inequality with respect to both the variables $v$ and $x$. In fact, since $Z$ is a convex set, its support function $\delta^*(\cdot \mid Z)$ is a closed proper convex function. 
Moreover, if $f(a,\cdot)$ is convex for all~$a$ 
then $f_*(v,x) = \inf_{a\in\R^m} \{ a^\T v -f(a,x) \}$ is concave in $(v,x)$ and therefore the claim holds.
Another advantage of Theorem~\ref{Ben_dual} is that
the computations of the support function $\delta^*$ 
and the conjugate function $f_*$ are independent.
This independence enables us to compute them easily, even if the constraint function $f$ or the structure $Z$ of the uncertainty set is changed.

In this section, we restrict ourselves to the standard robust constraint
$\max_{a\in U} f(a,x)\le 0$. Other equivalent sign conventions, such as
$\min_{a\in U} f(a,x)\ge 0$, will be discussed later when they are more convenient for the application under consideration.

\subsection{Necessity of positive-valued uncertainty sets}
As mentioned in Section \ref{sect:intro},
optimization problems in applications often involve
parameters that take only positive values. 
However, existing models of uncertainty sets are not sufficient
to represent the uncertainty of such parameters.  
For example, let us consider a box or an ellipsoidal uncertainty set. 
Note that the size of the uncertainty set reflects the level of robustness.
When the size of a box or an ellipsoidal uncertainty set is large, it can contain points with negative components, as shown in Figure \ref{fig:ellip}.
This means that, under a high level of robustness or when the data uncertainty, such as variance, is large, these uncertainty sets may fail to preserve the positivity of the uncertain parameters.
This can result in returning overly conservative robust optimal solutions 
or, in some cases, even make the robust counterpart infeasible.

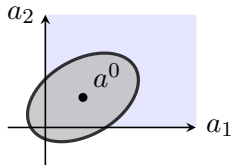
\begin{figure}[ht]
    \centering
    \begin{tikzpicture}
    \filldraw[fill=blue, opacity=0.1] (0.0001,1.5)--(0.0001,0.0001)--(2,0.0001)--(2,1.5);
    \coordinate (a^0) at (0.5,0.4); 
    \filldraw[fill=lightgray, very thick, , opacity=0.8] 
        (a^0) circle [x radius=0.8, y radius=0.5, rotate=30];
    \fill (a^0) circle (0.06) node[above right]{$a^0$};
    \coordinate (X_min) at (-0.5,0); 
    \coordinate (X_max) at (2,0); 
    \coordinate (Y_min) at (0,-0.5); 
    \coordinate (Y_max) at (0,1.5); 
    \draw[->,>=stealth,semithick] (X_min)--(X_max)node[right]{$a_1$}; 
    \draw[->,>=stealth,semithick] (Y_min)--(Y_max)node[left]{$a_2$}; 
    \end{tikzpicture}
    \caption{``Nonpositivity'' in an existing uncertainty model}
    \label{fig:ellip}
\end{figure}

\begin{example} \label{ex:infeas1}
    Consider the nominal problem with the feasible set 
    \begin{equation*}
    S(a, b) = \left\{ x \in \Rn \bigmid a^{\T} x \geq b, \ x \geq 0 \right\},
    \end{equation*}
    with uncertain parameter vector $a \in \R^n_{++}$ 
    and fixed parameter $b \in \R_{++}$.
    For the uncertainty set, let us use the ellipsoidal model defined as 
    \begin{equation} \label{eq:ex_ellip}
        \mathcal{U} \coloneqq \left\{ a \in \R^n \bigmid 
        (a - a^0)^{\T} \Sigma^{-1} (a - a^0) \leq \delta^2 \right\},
    \end{equation}
    where $a^0 \in \R^n$ denotes the nominal value of $a$, 
    and $\Sigma \in \R^{n \times n}$ is a positive definite matrix.
    Then the robust constraint is written as
    \begin{equation} \label{ineq:cons_example_nonpos}
        x \in S(a, b), \ \forall a \in \mathcal{U}.
    \end{equation}
    Note that $\delta \geq 0$ is a model parameter that determines to 
    the size of $\mathcal{U}$.
    In this situation, if we set $\delta > \| \Sigma^{-1/2} a^0 \|_2$, 
    then the ellipsoid $\mathcal{U}$ contains the origin in its interior 
    and thus also {contains} a point $a^- \in - \R^n_{++}$.
    This implies that there does not exist {any} $x$ such that $x \in S(a^-, b)$, 
    since $b > 0$, and therefore the robust constraint~\eqref{ineq:cons_example_nonpos} 
    has no robust feasible solutions.
\end{example}
\begin{example}
    Consider the nominal problem of the form
    \begin{equation*}
             \begin{aligned}
             &{\displaystyle \max_x} \quad a^{\T} x \\
             & \ {\rm s.t.} \hspace{5mm} x \in S, \quad x \geq 0, 
             \end{aligned}
    \end{equation*}
    where $a \in \Rn_{++}$ is the uncertain parameter vector 
    and $S \subseteq \Rn$ is a fixed nonempty set satisfying $0 \in S$.
    Then the robust counterpart is written as
    \begin{equation*}
             \begin{aligned}
             &{\displaystyle \max_x} \quad  \min_{a \in \mathcal{U}} \ a^{\T} x \\
             & \ {\rm s.t.} \hspace{5mm} x \in S, \quad x \geq 0.
             \end{aligned}
    \end{equation*}
    If we use the ellipsoidal uncertainty set as~\eqref{eq:ex_ellip} 
    and $\delta > \| \Sigma^{-1/2} a^0 \|_2$, 
    then since $x \geq 0$ we have $\min_{a \in \mathcal{U}} \ a^{\T} x \leq 0$.
    Hence, it yields the robust solution $x = 0$, 
    indicating that the robust counterpart has become meaningless, 
    because the positivity of the parameter was lost.
\end{example}
Therefore, in addition to tractability, we need an uncertainty model that can deal with positive-valued parameters, in the sense that it preserves ``positivity'', that is, the uncertainty set remains in $\R^m_{++}$ for any level of uncertainty.


\section{Uncertainty model 
     for positive-valued parameters} \label{sect:log}
In this section, we propose a new uncertainty set 
to deal with positive-valued parameter uncertainty.
Moreover, we derive an equivalent dual expression 
of the robust counterpart associated with our proposed model. 
\subsection{A new uncertainty model}
To handle the uncertainty of the parameter vector $a \in \R^m$ 
of~\eqref{constraint:robust_U}, we propose the following uncertainty set. 
\begin{dfn}
     We define a set $\Omega(a^0, \tau, A)$
     as follows:
     \begin{equation}
          \Omega(a^0, \tau, A) \coloneqq 
          \left\{ 
               a \in \R^m \bigmid
               a = A^0 z, \ 
               z_i - \ln z_i -1 \leq y_i \ (i = 1, \dots, m), \
               \| A y \| \leq \tau
          \right\},  
          \label{eq:omega}
     \end{equation}
     where 
     $a^0 \in \R^m_{++}$, $\tau \geq 0$, $A \in \R^{m \times m}$
     are parameters that determine the shape of the set, 
     and the matrix $A^0 \in \R^{m \times m}$ is defined as
     $A^0 \coloneqq \diag \left(a^0 \right)$. Recalling the robust constraint~\eqref{constraint:robust_U}, we also assume that $a^0 \in \text{\rm ri dom}(f(\cdot, x))$ for all $x$.
\end{dfn}

Clearly, the proposed uncertainty set preserves the positivity of the uncertain parameter vector~$a$, i.e., $\Omega(a^0, \tau, A) \subset \R^m_{++}$, due to the definition of~$A^0$ and~$z\in \R^m _{++}$.
Meanwhile, the set also measures the variation of the uncertain parameter~$a$ 
from a point $a^0 \in \R^m_{++}$ by means of the convex function $g \colon \R_+ \to \R$ defined as
\begin{equation} \label{func:g}
  g(t) \coloneqq  t - \ln t - 1,  
\end{equation}
and contains all points whose variation is less than or equal to~$\tau$.
Observe that the function $g$ satisfies $g(t) \geq 0$ for all $t > 0$, 
and $g(t) = 0$ if and only if $t = 1$.
Moreover, the value of $g(t)$ increases as $t$ moves away from $1$, 
and approaches infinity as $t \to \infty$.
Due to these properties, $g$ can be regarded as a function 
which measures the variation of $t$ from $1$.

We present a detailed interpretation in the simple case of $m = 1$. In this case, $\Omega(a^0, \tau, A)$ is a one-dimensional set
\begin{align} \label{eq:omega_1dim}
     \Omega(a^0, \tau, A) = \left\{ 
          a \in \R \bigmid z = \frac{a}{a^0}, 
          \ A(z - \ln z - 1) \leq \tau 
          \right\}.
\end{align}
Note that here $A$ is a scalar and therefore both sides of $A(z - \ln z - 1) \leq \tau$ are scaled
by a nonnegative constant.
Thus, without loss of generality, we can set $A = 1$.
Figure~\ref{fig:model_description} shows the set $\Omega(a^0, \tau, 1)$ when $m = 1$.
First, the relation $z = a / a^0$ in~\eqref{eq:omega_1dim} shifts
the center of the uncertainty of~$a$, the point from which the variation is measured, from~$a^0$ to~$1$.
Second, the variation of the shifted parameter~$z$
from~$1$ is expressed via~$g(z)$.
Third, the inequality $z - \ln z - 1 \leq \tau$ means that 
the interval defined by~$g(z) \leq \tau$ is the set of values of~$z$
for which $a = a^0 z \in \Omega(a^0, \tau, 1)$.
Finally, the uncertainty set $\Omega(a^0, \tau, 1)$ is obtained
as an interval of $a \in \R_{++}$
by scaling the interval of $z$ by $a^0$.

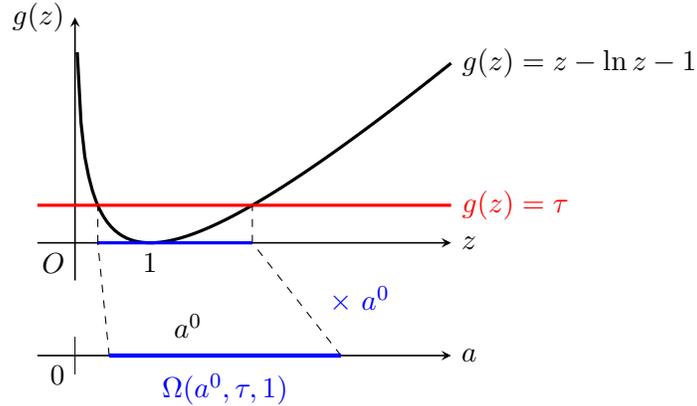
\begin{figure}[ht] 
    \centering
    \begin{tikzpicture}
    \coordinate (X_min) at (-0.5,0); 
    \coordinate (X_max) at (5,0); 
    \coordinate (Y_min) at (0,-0.5);
    \coordinate (Y_max) at (0,3); 
    \draw[->,>=stealth,semithick] (X_min)--(X_max)node[right]{$z$}; 
    \draw[->,>=stealth,semithick] (Y_min)--(Y_max)node[left]{$g(z)$}; 
    \draw (0,0)node[below left]{$O$}; 
    \draw (1,0)node[below]{$1$}; 
    \draw[very thick,samples=100,domain=0.03:5, name path=g] 
        plot(\x,{\x - ln(\x) - 1})node[right]{$g(z) = z - \ln z - 1$};
    \draw[very thick, red, name path=tau] (-0.5,0.5)--(5,0.5)node[right]{$g(z) = \tau$};
    \path[name intersections={of= g and tau, by={A,B}}]; 
    \coordinate (A_z) at ($(X_min)!(A)!(X_max)$); 
    \coordinate (B_z) at ($(X_min)!(B)!(X_max)$); 
    \draw[dashed] (A_z)--(A);
    \draw[dashed] (B_z)--(B);
    \draw[very thick, blue] (A_z)--(B_z); 
    \coordinate (X_min_3) at ($(X_min) - (0,1.5)$); 
    \coordinate (X_max_3) at ($(X_max) - (0,1.5)$); 
    \coordinate (Y_min_3) at (0,-1.5-0.25); 
    \coordinate (Y_max_3) at (0,-1.5+0.25);
    \draw[->,>=stealth,semithick] (X_min_3)--(X_max_3)node[right]{$a$}; 
    \draw (Y_min_3)--(Y_max_3);
    \draw ($(1.5,-1.5) + (0,0.1)$)node[above]{$a^0$}; 
    \draw (0,-1.5)node[below left]{$0$}; 
    \tikzmath{
        coordinate \c; \c{A_z} = (A_z);
        coordinate \c; \c{B_z} = (B_z);
    }
    \coordinate (A_a) at (1.5 * \cx{A_z}, -1.5); 
    \coordinate (B_a) at (1.5 * \cx{B_z}, -1.5);
    \draw[dashed] (A_a)--(A_z);
    \draw[dashed] (B_a)--(B_z);
    \draw[ultra thick, blue] (A_a)--(B_a);
    \draw[blue] ($(A_a)!0.5!(B_a) + (0, -0.1)$)node[below]{$\Omega(a^0, \tau, 1)$};
    \draw[blue] ($(B_z)!0.5!(B_a) + (0.3, 0)$)node[right]{$\times \ a^0$};
    \end{tikzpicture}
    \caption{The model description in the case of $m = 1$}
    \label{fig:model_description}
\end{figure}

We present the following remarks regarding our uncertainty set~\eqref{eq:omega}.
\begin{enumerate}
     \item The parameter~$\tau$ determines 
          the size of the uncertainty set \eqref{eq:omega_1dim}.
          Indeed, $\Omega(a^0, \tau, 1)$ becomes larger with the center~$a^0$ 
          as~$\tau$ increases.
          As we show in Proposition~\ref{prop:property_basic},
          $\tau$~has the same meaning when $m \geq 2$.

     \item The idea for the case $m \geq 2$ is the same as the case $m = 1$: 
          we transform~$a$ by~$a^0$,  
          measure the variation of~$z$ from~$1$ using the function~$g$, 
          and include all points whose values of $g(z)$ {are} less than or equal to~$\tau$.
          The difference is that {for} $m \geq 2$,
          the inequality $z - \ln z - 1 \leq \tau$ is replaced 
          by $\| A {\tilde{g}} \| \leq \tau$,
          where ${\tilde{g} \colon \R^m \to \R^m}$ is defined as ${\tilde{g}}(z) \coloneqq (g(z_1), \dots, g(z_m))^{\T}$.
          This allows us to incorporate the correlation between the components~$a_i$ (or $z_i$) 
          in the shape of the uncertainty set~\eqref{eq:omega}.
          The matrix~$A$ models the correlation of~$a_i$.
          For example, we can use the covariance matrix of~$a$, 
          like an ellipsoidal uncertainty set.

     \item In the definition of our model~\eqref{eq:omega}, 
          we use the inequality $g(z) \leq y$ instead of the equation $g(z) = y$. This ensures that the set $\Omega(a^0, \tau, A)$ remains convex by avoiding the use of the nonconvex equality restriction $g(z) = y$.
          
     \item The function~$g$ appears in some contexts of optimization theory,
          such as the proof of the uniqueness of the central path in interior-point methods, 
          and the characterization of approximation matrices generated by the DFP formula in quasi-Newton methods. 
          The function~$g$ is suitable for measuring 
          the variation of positive-valued parameters~$z$ from~$1$ 
          because it is defined only for positive $t$, 
          and satisfies $\lim_{t \downarrow 0} g(t) = \infty$. 
          This ensures the uncertainty set remains in $\R^m_{++}$ 
          even for large~$\tau$, so that we can model the uncertainty of positive-valued parameters at any level of uncertainty.
          Moreover, the conjugate function of~$g$ can be written explicitly, which allows us to obtain an explicit dual expression~\eqref{constraint:Ben_dual} for our model.
\end{enumerate}

The following proposition shows that  
$\tau$ is a parameter that determines the size of the uncertainty set \eqref{eq:omega} 
not only when $m = 1$ but also in the case of $m \geq 2$. 
This confirms that~$\tau$ corresponds to the level of uncertainty assumed by the model.
\begin{prop} \label{prop:property_basic}
     Assume that $A$ is a non-singular matrix.
     Then the uncertainty set \eqref{eq:omega} satisfies
     \begin{enumerate}
          \item $\Omega(a^0, 0, A) = \{a^0\}$,
          \item $\tau_1 \leq \tau_2 \implies 
               \Omega(a^0, \tau_1, A) \subseteq \Omega(a^0, \tau_2, A)$,
     \end{enumerate}
     for arbitrary vector $a^0 \in \R^m_{++}$.
\end{prop}
\begin{proof}
     If $\tau = 0$ then \eqref{eq:omega} is given by
     \begin{equation*}
          \Omega(a^0, 0, A) = 
          \left\{ 
               a \in {\R^m} \bigmid 
               a = A^0 z, \ 
               z_i - \ln z_i -1 \leq y_i \ (i = 1, \dots, m), \
               \| A y \| \leq 0 \
          \right\}  .
          \label{omega_V=0}
     \end{equation*}
     Since $\| A y \| \leq 0$ and $A$ is {non-singular}, 
     we obtain $y = 0$, 
     which implies $z_i - \ln z_i -1 \leq 0, \ i = 1, \dots, m$.
     Recall that the function $g$ 
     satisfies $g(t) \geq 0$ for all $t > 0$
     and {becomes} 0 only when {$t = 1$}.
     Thus $z_i = 1$ {for} $i = 1, \dots, m$, and we obtain $a = {A^0} e = a^0$, 
     which {proves} property~(i).
     
     To show~(ii), let $a^1 \in \Omega(a^0, \tau_1, A)$.
     Then there exist $z^1, y^1 \in \R^m$ satisfying
     \begin{align*}
          a^1 = A^0 z^1, & \\
          z^1_i - \ln z^1_i - 1 \leq y^1_i, & \:\: i = 1, \dots, m, \\
          \| A y^1 \| \leq \tau_1. &
     \end{align*}
     Since $\tau_1 \leq \tau_2$, $y^1$ satisfies $\| A y^1 \| \leq \tau_2$, 
     and thus $a^1 \in \Omega(a^0, \tau_2, A)$. 
     Therefore, we have $\Omega(a^0, \tau_1, A) \subseteq \Omega(a^0, \tau_2, A)$.
     %
\end{proof}

The first property of the above proposition means that $\Omega$ is a singleton at~$a^0$ when~$\tau = 0$, 
which shows that~$\Omega$ is centered at the nominal value~$a^0$.
The second property indicates that~$\Omega$ becomes {larger} 
as~$\tau$ increases.
Figure~\ref{fig:unc} shows plots of our model \eqref{eq:omega} 
for several {values} of $\tau$ in the case $m = 2$.
We can also see in this figure that the two properties hold.
Thus, $\tau$ corresponds to the level of uncertainty assumed by the model.

\begin{figure}[ht]
     \begin{minipage}{.3\textwidth}
          \centering
          \includegraphics[width=4cm]{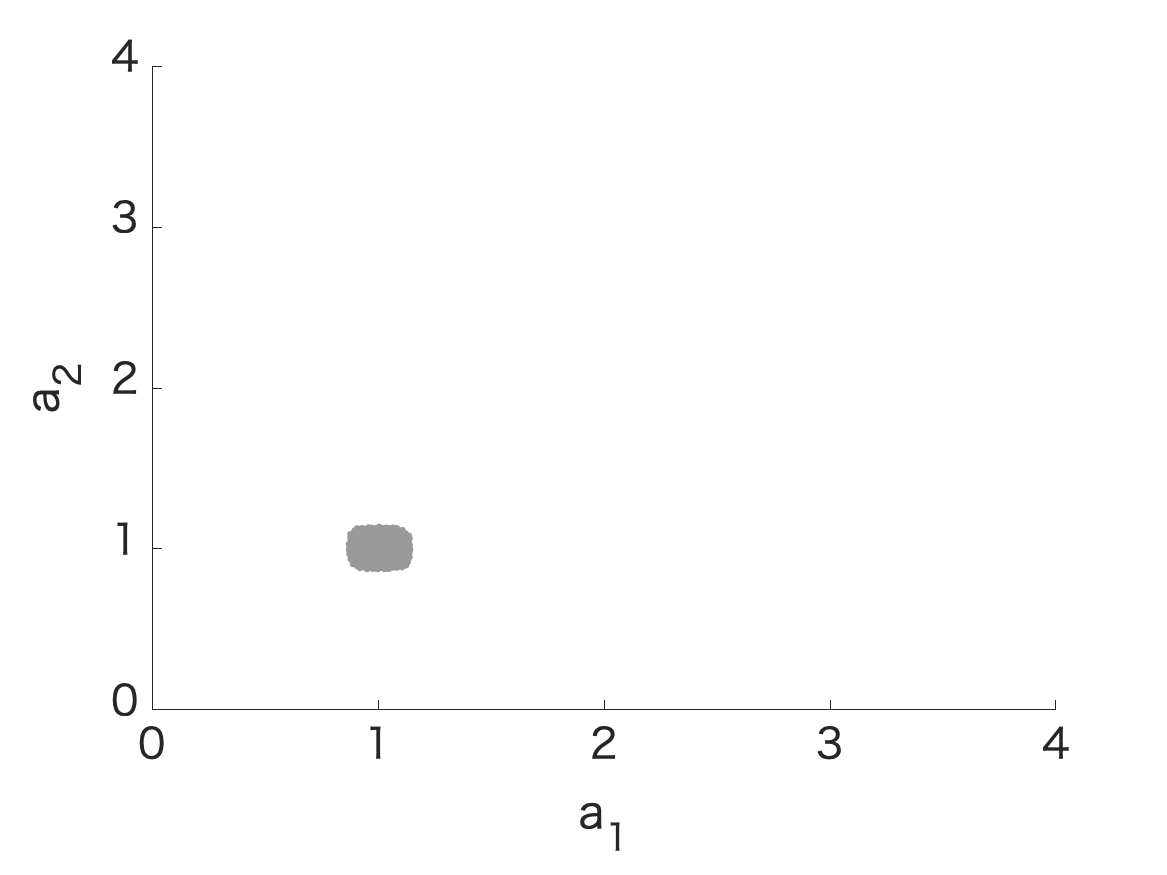}
          \subcaption{$\tau = 1 \times 10^{-2}$}
     \end{minipage}
     \begin{minipage}{.3\textwidth}
          \centering
          \includegraphics[width=4cm]{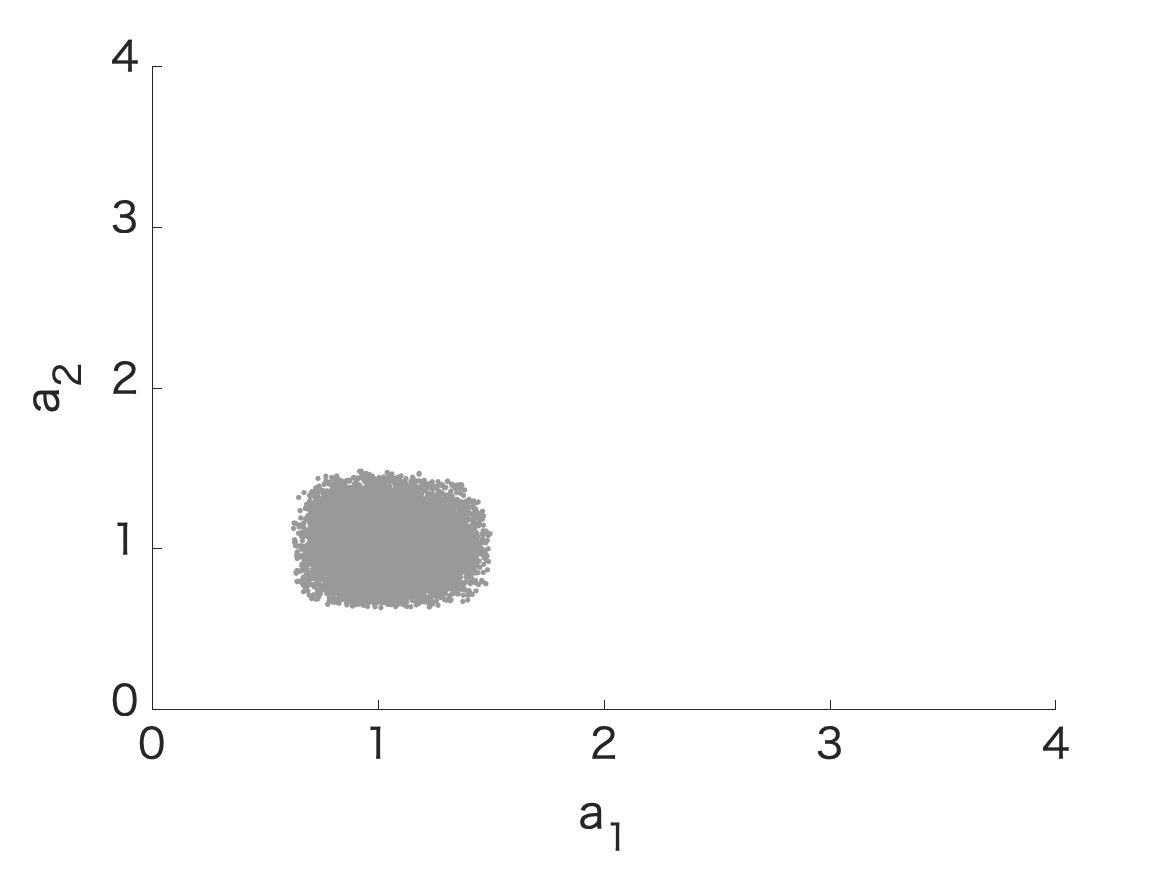}
          \subcaption{$\tau = 0.1$}
     \end{minipage}
     \begin{minipage}{.3\textwidth}
          \centering
          \includegraphics[width=4cm]{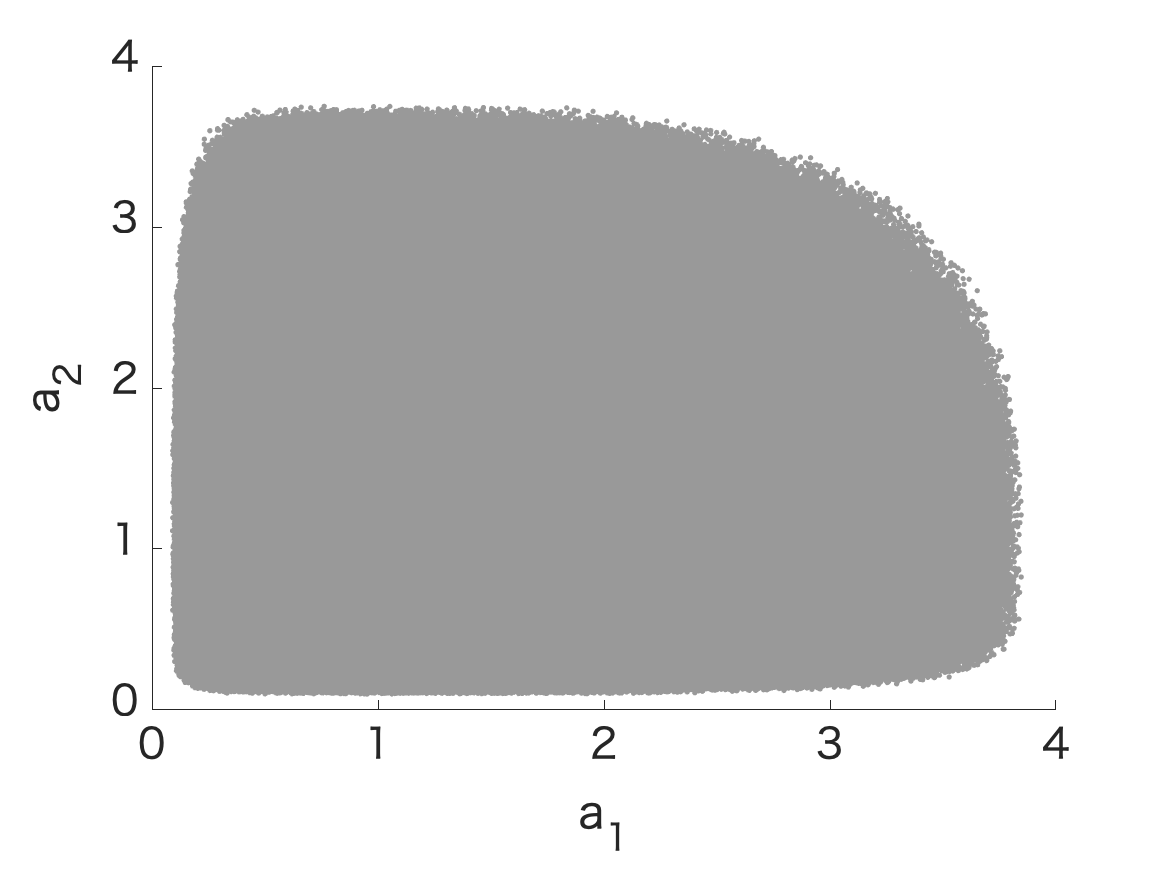}
          \subcaption{$\tau = 1.5$}
     \end{minipage}
     \caption{The proposed uncertainty set}
     \label{fig:unc}
\end{figure}

\begin{remark}
     As a generalized version of the proposed model \eqref{eq:omega}, 
     we can consider the following uncertainty set:
     \begin{equation} \label{eq:omega_general}
          \Omega(a^0, \tau, \delta, A, V) = 
          \left\{ 
               a \in {\R^m} \bigmid
               \begin{aligned}
                    a = A^0 z, \ 
                         z_i - \ln z_i -1 \leq y_i \ (i = 1, \dots, m), \\
                    \| A y \| \leq \tau, \
                         \| V(z-e) \| \leq \delta
               \end{aligned}
          \right\}, 
     \end{equation}
     where the matrix $V \in \R^{m \times m}$ and the scalar $\delta \geq 0$
     are additional parameters determining the shape of the set.
     Note that if the $\ell_2$-norm is used, the inequality $\| V (z - e) \| \leq \delta$ {describes} an {ellipsoidal region}. 
     {Thus,} \eqref{eq:omega_general} is a model that includes 
     an ellipsoidal {component in the description of the} uncertainty set.
     If $A = O$, \eqref{eq:omega_general} is equivalent to the ellipsoidal model 
     centered at the nominal value $a^0$. 
     {When} $V = O${,} \eqref{eq:omega_general} {reduces to the model}
    \eqref{eq:omega}.
\end{remark}
%


\subsection{Derivation of the tractable dual expression} \label{sect:dual}
In this subsection, we reformulate the robust convex counterpart 
obtained by using our proposed uncertainty model 
into an equivalent and tractable form using Fenchel duality.
We derive the dual expression of the general model~\eqref{eq:omega_general}.

\subsubsection{Rewriting for the dual expression-result}
In order to apply Theorem \ref{Ben_dual} to our model, 
we {rewrite} \eqref{eq:omega_general} {in the form consistent with the uncertainty set~\eqref{Ben_uncertainty_set} in Theorem~\ref{Ben_dual}} as follows:
\begin{equation} \label{eq:rewrite_Omega}
    {\Omega (a^0, \tau, \delta, A, V)} = \left\{ a = \mat{O & A^0} \mat{\eta \\{\xi}} \right.
     \left.\bigmid {\zeta} = \mat{\eta \\{\xi}} \in Z^0 \right\}, 
\end{equation}
where $\eta, \xi \in \R^m$, $A^0 \in \R^{m \times m}$, 
$O \in \R^{m \times m}$ denotes the zero matrix, 
$Z^0 \coloneqq Z^0_1 \cap Z^0_2 \cap Z^0_3$, and $Z^0_1$, $ Z^0_2$, $Z^0_3$ are defined respectively as
\begin{align*}
     Z^0_1 &\coloneqq \left\{ {\zeta} = \mat{\eta \\{\xi}} \in \R^{2m}
          \bigmid{\xi}_i - \ln{\xi}_i - 1 - \eta_i \leq 0
          \ (i = 1, \dots, m) \right\}, \\
     Z^0_2 &\coloneqq \left\{ {\zeta} = \mat{\eta \\{\xi}} \in \R^{2m}
          \bigmid \|A\eta\| \leq \tau \right\}, \\
     Z^0_3 &\coloneqq \left\{ {\zeta} = \mat{\eta \\{\xi}} \in \R^{2m} 
          \bigmid \|V({\xi} - e)\| \leq \delta \right\}.
\end{align*}
However, we cannot apply Theorem \ref{Ben_dual} directly by replacing~\eqref{Ben_uncertainty_set} with~\eqref{eq:rewrite_Omega}
because the assumption $0 \in \text{ri } {Z^0}$ is not satisfied.
Thus, we now show that we can still obtain the result of Theorem~\ref{Ben_dual} 
by shifting $Z^0$ in parallel, 
even when $0 \notin \text{ri } Z^0$.
\begin{cor} \label{cor:Ben_dual_arg}
     Suppose that the function $f \colon \R^m \times \R^n \to \R$
     is such that $f(\cdot,x)$ is concave for all~$x$
     and assume that the nominal value $\tilde{a}^0$ satisfies
     $\tilde{a}^0 + \tilde{B} \zeta^0 \in \text{\rm ri dom}(f(\cdot, x))$
     for all~$x$.
     Let $\tilde{\mathcal{U}}$ be an uncertainty set defined as  
     \begin{equation*}
          \tilde{\mathcal{U}} \coloneqq \left\{ a \in \R^m \bigmid 
               a = \tilde{a}^0 + \tilde{B}{\zeta}, \ {\zeta} \in \tilde{Z}^0 \right\},
     \end{equation*}
     where $\tilde{B} \in \R^{m \times l}$ and $\tilde{Z}^0 \subset \R^l$ is a nonempty compact convex set with
     ${\zeta}^0 \in \text{\rm ri } \tilde{Z}^0$.
     Let $\tilde{Z} \coloneqq \tilde{Z}^0 - \{{\zeta}^0\}$.
     Then $\tilde{\mathcal{U}}$ can be rewritten as
     \begin{equation*}
     \tilde{\mathcal{U}} = \left\{ a \in \R^m \ \big| \ a= (\tilde{a}^0 + \tilde{B} \zeta^0) + \tilde{B} \zeta', \ \zeta' \in \tilde{Z} \right\},
     \end{equation*}
     and $x \in \Rn$ satisfies $\max_{a \in \tilde{\mathcal{U}}} f(a,x) \le 0$ if and only if there exists a vector $v \in \R^m$
     such that
     \begin{equation} \label{Ben_dual_arg}
          (\tilde{a}^0 + \tilde{B} {\zeta}^0)^{\T} v + \delta^*(\tilde{B}^{\T} v \mid \tilde{Z}) 
          - f_*(v, x) \leq 0.
     \end{equation}
\end{cor}
\begin{proof}
     First, note that ${\zeta}^0 \in \text{\rm ri } \tilde{Z}^0$ implies $0 \in \text{\rm ri } \tilde{Z}$.
     Also, since
     $$\tilde{Z}^0 = \tilde{Z} + \{{\zeta}^0\}= \left\{ {\zeta} \in \R^l \mid
     {\zeta} = {\zeta}' + {\zeta}^0 , \ {\zeta}' \in \tilde{Z} \right\},$$
     we can rewrite $\tilde{\mathcal{U}}$ in terms of $\tilde{Z}$ as follows:
     \begin{align*}
          \tilde{\mathcal{U}} &= \left\{ a \in \R^m \bigmid a= \tilde{a}^0 + \tilde{B} {\zeta},
               \ {\zeta} \in \tilde{Z}^0 \right\} \\
          &= \left\{ a \in \R^m \bigmid a= \tilde{a}^0 + \tilde{B} ({\zeta}' + {\zeta}^0),
               \ {\zeta}' \in \tilde{Z} \right\} \\
          &= \left\{ a \in \R^m \ \big| \ a= (\tilde{a}^0 + \tilde{B} {\zeta}^0) + \tilde{B} {\zeta}',
               \ {\zeta}' \in \tilde{Z} \right\}.
     \end{align*}
     Thus, we obtain (\ref{Ben_dual_arg}) 
     by using Theorem~\ref{Ben_dual}.
\end{proof}

We now apply Corollary~\ref{cor:Ben_dual_arg} to rewrite
the proposed uncertainty set~\eqref{eq:rewrite_Omega} as follows.
We obtain $\tilde{\mathcal{U}} = \Omega$ by taking
$\tilde{a}^0 = 0$, $\tilde{B} = \mat{O & A^0} \in \R^{m \times 2m}$ and $\tilde{Z}^0 = Z^0$.
{Assuming $\tau > 0$ (since for $\tau = 0$, the robust constraint trivially reduces to the nominal constraint $f(a_0,x) \leq 0$),}
let $\zeta^0 = \mat{{\varepsilon e} \\ e} \in \R^{2m}$.  {From the definition of $Z_2^0$, by choosing  $\varepsilon > 0$ sufficiently small such that $\varepsilon \|Ae\| < \tau$},
we guanrantee $\zeta^0 \in \text{\rm ri } Z^0$ 
from~\cite[Theorems~6.5 and~6.8]{Roc70},
which implies that $0 \in \text{\rm ri } Z$ with $Z \coloneqq Z^0 - \{\zeta^0\}$.
Observe also that in this case $\tilde{a}^0 + \tilde{B} \zeta^0 = A^0 e = a^0 \in \text{\rm ri dom}(f(\cdot, x))$ for all $x$.
Then~\eqref{eq:rewrite_Omega} can be rewritten as 
\begin{equation}\label{eq:omega_use_cor}
     {\Omega (a^0, \tau, \delta, A, V)} = \left\{ {a} = a^0 + \mat{O & A^0} \mat{\eta \\{\xi}} \bigmid
     {\zeta} = \mat{\eta \\{\xi}} \in Z \right\}, 
\end{equation}
where 
\begin{align}
     Z & \coloneqq Z_1 \cap Z_2 \cap Z_3, 
          \label{Z_intersection} \\
     Z_1 &\coloneqq \left\{ {\zeta} = \mat{\eta \\{\xi}} \in \R^{2m} \bigmid
          ({\xi}_i+1) - \ln ({\xi}_i+1) - 1 - \eta_i { - \varepsilon } \leq 0 \
          (i = 1, \dots, m) \right\}, 
           \label{Z_1} \\
     Z_2 &\coloneqq \left\{ {\zeta} = \mat{\eta \\{\xi}} \in \R^{2m} \bigmid
          \|A(\eta { + \varepsilon e})\| \leq \tau \right\}, 
          \label{Z_2} \\
     Z_3 &\coloneqq \left\{ {\zeta} = \mat{\eta \\{\xi}} \in \R^{2m} \bigmid
          \|V{\xi} \| \leq \delta \right\}.
          \label{Z_3}
\end{align}
\subsubsection{Computation of the support function}
We now compute the support function $\delta^*(y \mid Z)$ 
of the set $Z$ in \eqref{eq:omega_use_cor},
which is required to obtain the dual expression~\eqref{Ben_dual_arg}.
By~\cite[Lemma 6.4]{Ben15}, 
$\delta^*(y \mid Z)$ can be expressed as the sum of the support functions 
of the three sets $Z_1$, $Z_2$, and $Z_3$.
The following three lemmas are the results 
of the calculation of each support function.
%

\begin{lemma} \label{lemma_delta^*_1}
     The support function $\delta^*(y \mid Z_1)$ of the set $Z_1$ 
     defined in \eqref{Z_1} is given as follows
     \begin{equation} \label{eq:delta^*_1}
          \delta^*(y \mid Z_1) = \min_{
          \substack{
          {u \in \R^m_{++}}, \\
          w \in \R^m}
          } 
          \left\{ \sum_{k = 1}^m 
          \left(\varepsilon u_k -  u_k \ln (1 - \frac{w_k}{u_k} ) - w_k \right) \bigmid
          y = {\mat{-u\\ w }}, \ 
          {u-w \in \R^m_{++} } \right\} .
     \end{equation}
\end{lemma}
\begin{proof}
     We can represent $Z_1$ in terms of $m$ convex functions as follows:
     \[ 
          Z_1 = \left\{ {\zeta} \in \R^{2m} \bigmid
          h_k({\zeta}) \leq 0, \: k = 1, \dots, m \right\}, 
     \]
     with
     \[
          h_k({\zeta}) \coloneqq ({\zeta}_{m+k} + 1) - \ln ({\zeta}_{m+k} + 1) - 1 - {\zeta}_k { - \varepsilon}, \quad 
               k = 1, \dots, m. 
     \]
     Moreover, for each $k = 1, \dots, m$, the function $h_k$ can be written as
     \[ h_k({\zeta}) = {h_\eta(\zeta_k) + h_\xi (\zeta_{m+k})}, \]
     where $h_\eta \colon \R \to \R$ and $h_\xi \colon \R \to \R$ are defined respectively as
     \[
          h_{{\eta}} (t) \coloneqq - t { - \varepsilon}, \quad 
          h_{{\xi}} (t) \coloneqq (t + 1) - \ln (t + 1) - 1.
     \]
     Then from the result for the ``separable case'' in{~\cite[Lemma 9]{Ben15}},
     the support function $\delta^*(y \mid Z_1)$ is given by
     \begin{equation} \label{delta^*}
          \delta {^*} (y \mid Z_1) = \min_{
          \substack{
          {u \in \R^m_{++}}, \\
          W \in \R^{m \times 2m}}
          } 
          \left\{ 
               \sum_{k = 1}^m
               {u_k \left( h_{\eta}^* (\frac{W_{k,k}}{u_k} ) + h_{\xi}^* (\frac{W_{k,m+k}}{u_k} ) \right)}
               \bigmid 
               {
               W = \mat{\diag(y^{(1)}) & \diag(y^{(2)})}}
        \right\}
     \end{equation}
     where $y^{(1)}\coloneqq [y_1,\ldots,y_m]^\T \in \R^m$ and $y^{(2)}\coloneqq [y_{m+1},\ldots,y_{2m}]^\T \in \R^m$.
     Hence, we need to derive the conjugate functions {$h^*_{\eta}$ and $h^*_{\xi}$}. {We obtain}
     \begin{align*}
          h_{{\eta}}^*(s) &= \sup_{t \in \R} \{ s t + t {+ \varepsilon} \} \\
          &= 
          \begin{cases}
               {\varepsilon}, & s = -1, \\
               + \infty, & \mbox{otherwise},
          \end{cases}
     \end{align*}
     and 
     \begin{align*}
          h_{{\xi}}^*(s) &= \sup_{t > - 1} 
          \left\{ s t - ((t + 1) - \ln (t + 1) -1) \right\} \\
          &= 
          \begin{cases}
               - \ln (1 - s) - s, & s < 1, \\
               + \infty, & \mbox{otherwise},
          \end{cases}
     \end{align*}
     which from (\ref{delta^*}) implies
     \begin{equation} \label{delta^*_2}
          \delta^*(y \mid Z_1) = 
          \min_{
          \substack{
          {u \in \R^m_{++}}, \\
          {W \in \R^{m \times 2m}}
          }}
          \left\{- \sum_{k = 1}^m 
          u_k \left[ \ln \left(1 - \frac{{W_{k,m+k}}}{u_k} \right)
          + \frac{{W_{k,m+k}}}{u_k} {- \varepsilon} \right] \bigmid 
          {(} u, {W} {)} \in \mathcal{S} (y) 
          \right\}, 
     \end{equation}
     where the set $\mathcal{S} (y)$ is {defined} by 
     \begin{equation*}
          \mathcal{S} (y) = \left\{ 
               (u, W) \in \R^m \times \R^{m \times 2m} \bigmid 
               \begin{aligned}
                    & {
                    W = \mat{\diag(y^{(1)}) & \diag(y^{(2)})}},\\ 
                    &\frac{{W_{k,k}}}{u_k} = -1, \ \frac{{W_{k, m+k}}}{u_k} < 1, \quad 
                         k = 1, \dots, m
               \end{aligned}
          \right\}.
     \end{equation*}
     Note that $\mathcal{S} (y)$ {can be} rewritten as
     \begin{equation*}
          \mathcal{S} (y) = \left\{ 
               (u, W) \in \R^m \times \R^{m \times 2m} \bigmid 
               \begin{aligned}
                    y_l &= 
                    \begin{cases}
                         -u_l, &  l = 1, \dots, m \\
                         {W_{l - m,l}}, & l = m+1, \dots, 2m
                    \end{cases}, \\
                    u_k &> {W_{k, m + k}}, \ k = 1, \dots, m.
               \end{aligned}
          \right\}.
     \end{equation*}
     Thus from (\ref{delta^*_2}) we have
     \begin{align*}
          \delta^*(y \mid Z_1) =
          \min_{
          \substack{
          {u \in \R^m_{++}}, \\
          {W \in \R^{m \times 2m}}
          }} & \left\{ 
               - \sum_{k = 1}^m
               \left[
                    u_k \ln \left(1 - \frac{{W_{k,m+k}}}{u_k} \right) { - \varepsilon u_k}
                    + {W_{k,m+k}} 
               \right] \bigmid 
          \right. \\
               & \quad y = \mat{-u^{\T}, {W_{1,m+1}}, \dots, {W_{m, 2m}} }^{\T}, \:\:
               {u_k} > {W_{k, m + k}}, \ k = 1, \dots, m
          \Bigg\},
     \end{align*}
     which implies (\ref{eq:delta^*_1})
     with $w \coloneqq \left( {W_{1,m+1}}, \dots, {W_{m, 2m}} \right)^{\T}$.
\end{proof}
%
%
\begin{lemma} \label{lemma_delta^*_2}
     The {support} function $\delta^*(y \mid Z_2)$ of the set $Z_2$ 
     defined in \eqref{Z_2} is given as follows:
     \begin{equation} \label{eq:delta^*_2}
          \delta^*(y \mid Z_2) = \inf_s 
          \left\{ \tau \| s \|_*  - \varepsilon e^\T A^\T s\bigmid
          \mat{A^{\T}\, \\ O} s = y \right\}
     \end{equation}
     where $y \in \R^{2m}$.
\end{lemma}
\begin{proof}
     {For any $\zeta = \mat{\eta \\ \xi} \in Z_2$, let $\tilde \zeta \coloneqq \mat{\tilde{\eta} \\ \xi}$ where $\tilde \eta \coloneqq \eta + \varepsilon e$. By defining $M \coloneqq \mat{A & O}$,
      we can express $Z_2$ as
     \[ 
          Z_2 = \left\{ \tilde{\zeta} - \mat{\varepsilon e \\ 0} \in \R^{2m} \bigmid
          \| M \tilde{\zeta} \| \leq \tau \right\}
     \]
     as a translation of a centered set $\tilde{Z}_2$ defined as
     $$\tilde{Z}_2 \coloneqq \left\{ \tilde{\zeta} \in \R^{2m} \bigmid
          \| M \tilde{\zeta} \| \leq \tau \right\}.
     $$
     }
     The definition of support function gives the relationship between the support functions of~$Z_2$ and~$\tilde{Z}_2$ as
     \begin{equation} \label{eq:translation_Z2}
         \delta^* (y \mid Z_2) = \delta^* \bigg(y \: \bigg| \: \tilde{Z_2} - \mat{\varepsilon e\\ 0} \bigg) = \delta^* (y \mid \tilde{Z_2}) - y^\T \mat{\varepsilon e \\ 0}.
     \end{equation}
     For the set $\tilde{Z}_2$, the indicator function is
     given by
     \begin{align*}
          \delta(\tilde{\zeta} \mid \tilde{Z}_2) &=
          \begin{cases}
               0, & \| M \tilde{\zeta} \| \leq \tau, \\
               +\infty, & \mbox{otherwise},
          \end{cases} \\
          &=  q(M \tilde{\zeta}),
     \end{align*}
     where
     \[ 
          q(\tilde{\zeta}) = 
          \begin{cases}
               0, & \|\tilde{\zeta} \| \leq \tau, \\
               +\infty, & \mbox{otherwise}.
          \end{cases}
     \]
     Therefore, from~\cite[Lemma 6.7]{Ben15}, the support function for $\tilde{Z}_2$ is
     \begin{align*}
          \delta^*(y \mid \tilde{Z}_2) &=
          \inf_s \left\{ q^*(s) \bigmid M^{\T} s = y \right\} \\
          &= \inf_s \left\{ \tau \| s \|_* \bigmid M^{\T} s = y \right\}
     \end{align*}
     Substituting the above equation back into \eqref{eq:translation_Z2}, we obtain
     \begin{align*}
          \delta^*(y \mid {Z}_2) &=
          \inf_s \left\{ \tau \|s \|_* - y^\T \mat{\varepsilon e\\ 0} \bigmid \mat{A^\T\, \\ O}  s = y \right\} \\
          &= \inf_s \left\{ \tau \| s \|_* - \varepsilon e^\T A^\T s \bigmid \mat{A^\T\, \\ O} s = y \right\},
     \end{align*}
     which is equivalent to (\ref{eq:delta^*_2}).
\end{proof}
%
%
\begin{lemma} \label{lemma_delta^*_3}
     The {support} function $\delta^*(y \mid Z_3)$ of the set $Z_3$ 
     defined in \eqref{Z_3} is given as follows:
     \begin{equation} \label{eq:delta^*_3}
          \delta^*(y | Z_3) =
          \inf_s \left\{ \delta \| s \|_* \bigmid
          \mat{O \\ V^\T} s = y \right\},
     \end{equation}
     where $y \in \R^{2m}$.
\end{lemma}
\begin{proof}
     Expression (\ref{eq:delta^*_3}) is easily verified in a similar way to
     (\ref{eq:delta^*_2}).
     Thus, $Z_3$ can be rewritten as
     \[ 
          Z_3 = \left\{ {\zeta} \in \R^{2m} \bigmid
          \left\| \mat{O & V} {\zeta} \right\| \leq \delta \right\},
     \]
     and then we have
     \[
          \delta^*(y \mid Z_3) =
          \inf_s \left\{ \delta \| s \|_* \bigmid N^{\T} s = y \right\},
     \]
     where $N = \mat{O & V}$.
\end{proof}
%

Using the formulas of the support function of~$Z$ given 
by Lemmas \ref{lemma_delta^*_1}--\ref{lemma_delta^*_3}, 
we can finally obtain the final dual expression.
\begin{thm} \label{thm_dual}
     Suppose that the function $f \colon \R^m \times \R^n \rightarrow \R$
     is such that $f(\cdot, x)$ is concave for all $x$, 
     and assume that the nominal value $a^0$ satisfies
     $a^0 \in \text{\rm ri dom}(f(\cdot, x))$
     for all~$x$.
     Then, $x$ satisfies the robust constraint with the uncertainty set \eqref{eq:omega_general}, i.e.,
     \begin{equation*}
          \max_{a \in \Omega(a^0, \tau,\delta, A, V)} f(a,x) \leq 0
     \end{equation*}
     if and only if 
     there exist vectors $u, v, w, s^1, s^2 \in \R^m$ that satisfy
     \begin{equation} \label{robustconst_dual}
          \begin{split}
               (a^0)^{\T} v
               - \sum_{k = 1}^m
                    \left\{ u_k \ln \left(1 - \frac{w_{k}}{u_k} \right) 
                    + w_{k} \right\}
               + \tau \| s^1 \|_*
               + &\delta \| s^2 \|_*
               - f_*(v, x) \leq 0 , \\
               A^{\T} s^1 &= u, \\
               w + V^{\T}s^2 &= A^0 v, \\
               u  & \in \R^m_{++}, \\
               u - w & \in \R^m_{++}.
          \end{split}
     \end{equation}
\end{thm}
\begin{proof}
     From Lemmas \ref{lemma_delta^*_1}--\ref{lemma_delta^*_3}
     and the definition of $Z$ in~\eqref{eq:omega_use_cor},
     and by using{~\cite[Lemma 6.4]{Ben15}}, 
     the support function of $Z$ is written as
     \begin{align}
          & \delta^*(y \mid Z) \notag \\
          = & \min_{y^1, y^2, y^3} \min_{
          \substack{
          {u \in \R^m_{++}},\\
          {u-w \in \mathbb{R}^m_{++}},\\
          w, s^1, s^2
          }} 
          \left\{ \varepsilon e^\T u  - \sum_{k = 1}^m
          \left\{ u_k \ln \left(1 - \frac{w_k}{u_k} \right) + w_k \right\}
          +\tau \| s^1 \|_* - \varepsilon e^\T A^\T s^1
          +\delta \| s^2 \|_*  \bigmid \right. 
               \notag \\ 
          & \hspace{90pt} y^1 + y^2 + y^3 = y, \ 
          y^1 = \mat{ - u \\ w }, \
          \mat{A^{\T} \\ O} s^1 = y^2, \ 
          \mat{O \\ V^\T} s^2 = y^3 \Bigg\} 
               \notag \\
          = & \min_{
          \substack{
          y^1, y^2, y^3,\\
          {u \in \R^m_{++}},\\
          {u-w \in \mathbb{R}^m_{++},}\\
           w, s^1, s^2
          }} \left\{ 
          \varepsilon e^\T (u - A^\T s^1) - \sum_{k = 1}^m
          u_k \left\{ \ln \left(1 - \frac{w_k}{u_k} \right) + w_k \right\}
          + \tau \| s^1 \|_*
          + \delta \| s^2 \|_*  \ \bigmid \right. 
               \notag\\ 
          & \hspace{52pt} \mat{A^{\T} s^1 - u  \\ w + V^{\T} s^2 } = y \Bigg\}
               \label{delta^*_Z}. 
     \end{align}
     Recall that in our formulation~\eqref{eq:omega_use_cor}, we set $\tilde{a}^0 = 0$, $\tilde{B} = \mat{O & A^0}$ and $\zeta^0 = \mat{\varepsilon e \\ e}$. It is then straightforward that 
     $$
     a^0 = \tilde{a}^0 + \tilde{B} \zeta^0 \in \text{\rm ri dom}(f(\cdot, x))
     $$
     for all~$x$. Hence, by applying Corollary~\ref{cor:Ben_dual_arg}, the robust constraint is satisfied if and only if there exists $v \in \R^m$ such that
     $$
     (a^0)^{\T} v + \delta^*(\tilde{B}^{\T} v \mid \tilde{Z}) 
          - f_*(v, x) \leq 0.
     $$
     To evaluate the support function, we substitute the input vector 
     $$y = \tilde{B}^\T v = \mat{O \\A^0} v= \mat{0 \\ A^0 v}.$$
     Substituting the $y$ into our block matrix equality constraint in~\eqref{delta^*_Z} yields 
     $$
     \mat{A^{\T} s^1 - u  \\ w + V^{\T} s^2 } = \mat{0 \\ A^0 v}.
     $$
     {Note that the first term in~\eqref{delta^*_Z} vanishes from the above equation.} By removing the minimum in~\eqref{delta^*_Z} due to the {\it general principle} in~\cite{Ben15}, we obtain the constraints
     \begin{align}
          (a^0)^{\T} v
          - \sum_{k = 1}^m
               \left\{ u_k \ln \left(1 - \frac{w_{k}}{u_k} \right) 
               + w_{k} \right\}
          + \tau \| s^1 \|_*
          + &\delta \| s^2 \|_*
          - f_*(v, x) \leq 0 ,
               \label{robustconst2_dual1} \\
          \mat{ - u + A^{\T} s^1 \\ w + V^{\T}s^2 } &= \mat{O \\ A^0}v, 
               \notag \\
          u  & \in \R^m_{++}, \\
          u - w & \in \R^m_{++},
               \notag
     \end{align}
     which are equivalent to 
     (\ref{robustconst_dual}).
\end{proof}
\begin{remark}
     If the function $f(a,\cdot)$ is convex for all $a \in \mathbb{R}^m$, 
     then the expanded dual expression~\eqref{Ben_dual_arg} 
     obtained in Corollary~\ref{cor:Ben_dual_arg} is also a convex constraint with respect to $(v,x)$.
     Thus, the dual expressions~\eqref{robustconst2_dual1} are also 
     convex constraints.
\end{remark}

%
%
\begin{example}[The case of linear constraint]
     Consider the case $f(a, x) = a^{\T} x - b$, with $b \in \R$.
     In this case $m = n$, and $f$ is a concave function for the decision variable $x$, 
     and a convex function for the parameter $a$.
     The concave conjugate function of $f$ is given as
     \begin{equation*}
          f_*(v, x) = 
          \begin{cases}
               b, & v = x, \\
               - \infty, & \mbox{otherwise}.
          \end{cases}
     \end{equation*}
     From Theorem \ref{thm_dual} we have the dual expression 
     \begin{align*}
          (a^0)^{\T} x - \sum_{k = 1}^n 
          \left\{ u_k \ln \left(1 - \frac{w_k}{u_k} \right) + w_k \right\}
          + \tau \| s^1 \|_*
          + &\delta \| s^2 \|_*
          \leq b, \\
          A^{\T} s^1 &= u, \\
          w + V^{\T}s^2 &= A^0 x, \\
          u & \in \R^m_{++}, \\
          u - w& \in  \R^m_{++},
     \end{align*}
     where $u, w, s^1, s^2 \in \Rn$.
     Moreover, if $V = O$ then these constraints are simplified as
     \begin{align*}
          - \sum_{k = 1}^n 
          u_k \ln \left(1 - \frac{a^0_k x_k}{u_k} \right)
          &+ \tau \| s^1 \|_*
          \leq b, \\
          A^{\T} s^1 &= u, \\
          u & \in \R^m_{++}, \\
          u - A^0 x& \in \R^m_{++}.
     \end{align*}
\end{example}
%
%

In the above, we have derived a tractable dual expression of a constraint with uncertain parameters.
Now we {extend} the result to a problem with uncertain parameters in the objective function:
\begin{equation}
    \label{prob:RC_P_obj}
    \tag{$\text{RC}_{\text{P obj}}$}
    \begin{aligned}
      &{\displaystyle \max_x} \quad  \min_{a \in \Omega} \ f_0(a, x) \\
      & \ {\rm s.t.} \hspace{5mm} x \in S,
    \end{aligned}
\end{equation}
where $S \subseteq \Rn$ is the feasible set of the nominal problem.

\begin{cor}
     Suppose that the function $f_0 \colon \R^m \times \R^n \to \R$
     is such that $f_0(\cdot, x)$ is convex for all $x$,
     and assume that the nominal value $a^0$ satisfies
     $a^0 \in \text{ri \ dom}(f_0(\cdot, x))$
     for all $x$.
     Then, \eqref{prob:RC_P_obj} is equivalent to the following problem:
     \begin{equation}
          \label{prob:dual_RC_P_obj}
          \tag{$\text{DRC}_{\text{P obj}}$}
          \begin{aligned}
               & \hspace{-3mm} {\displaystyle \max_{x, u, v, w, s^1, s^2}} \quad- (a^0)^{\T} v  
               + \sum_{k = 1}^n \left\{  u_k \ln \left(1 - \frac{w_{k}}{u_k} \right) + w_{k} \right\} 
               -
               \tau \| s^1 \|_*
               -
               \delta \| s^2 \|_*
               - f_0^*(- v, x) \\
               & \ {\rm s.t.} \hspace{9mm} x \in S, \\
               & \hspace{14mm} A^{\T} s^1 = u, \quad w + V^{\T}s^2 = A^0 v, \\
               & \hspace{14mm} {u \in \R^m _{++}}, \quad u - w  \in {\R^m _{++}.}
          \end{aligned}
     \end{equation}
\end{cor}
\begin{proof}
     In a similar way to the proof of~\cite[Theorem~2]{Ben15},
     we have
     \begin{align*}
          \min_{a \in \Omega} \ f_0(a, x) 
          &= \min_{a \in \R^m} \left\{ 
               f_0(a, x) + \delta(a \ | \ \Omega) \right\} \\
          &= \max_{v \in \R^m} \left\{
                - \delta^*( - v \ | \ \Omega) - f_0^*(v, x) \right\} \\
          &= \max_{v \in \R^m} \left\{
                - \delta^*( v \ | \ \Omega) - f_0^*(- v, x) \right\} \\
          &= \max_{v \in \R^m} \left\{
                - (a^0)^{\T} v - \delta^*(\tilde{B}^{\T} v \ |\  Z) - f_0^*(- v, x) \right\}, 
     \end{align*}
     where the second line follows from Fenchel duality.
     Note that strong duality holds 
     since $a^0 \in \text{ri dom}(f_0(\cdot, x))$ {for all} $x$ and 
     $a^0 \in \text{ri} \ \Omega$.
     Thus the equivalent problem to~\eqref{prob:RC_P_obj} is written as
     \begin{equation*}
          \begin{aligned}
               &{\displaystyle \max_{x, v}} \quad - (a^0)^{\T} v - \delta^*(\tilde{B}^{\T} v \ | \ Z) - f_0^*(- v, x) \\
               & \ {\rm s.t.} \hspace{5mm} x \in S.
          \end{aligned}
     \end{equation*}
     Then applying Lemmas \ref{lemma_delta^*_1}--\ref{lemma_delta^*_3}, 
     this is equivalent to~\eqref{prob:dual_RC_P_obj}.
     \end{proof}
\begin{example}[The case of linear objective function]
     If $f_0(a, x) = a^{\T} x$, then the dual problem associated with~\eqref{prob:RC_P_obj} can be written as follows:
     \begin{equation*}
          \begin{aligned}
               & \hspace{-3mm} {\displaystyle \max_{x, u, v, w, s^1, s^2}} \quad (a^0)^{\T} x
               + \sum_{k = 1}^n \left\{  u_k \ln \left(1 - \frac{w_{k}}{u_k} \right) + w_{k} \right\} 
               -
               \tau \| s^1 \|_*
               -
               \delta \| s^2 \|_* \\
               & \ {\rm s.t.} \hspace{9mm} x \in S, \\
               & \hspace{14mm} A^{\T} s^1 = u, \quad w + V^{\T}s^2 = - A^0 x, \\
               & \hspace{14mm} {u \in \R^m _{++}}, \quad u - w  \in {\R^m _{++}}.
          \end{aligned}
     \end{equation*}
\end{example}
%


\section{Properties of the robust counterpart with the proposed model} \label{sect:prop}
In this section, we derive some properties of the robust counterpart 
obtained by using the proposed uncertainty set \eqref{eq:omega}. 
{Those include}, for example, 
the boundedness and the asymptotic {behavior} of 
the worst-case value of the uncertain parameter $a$ 
and the risk-adjusted objective function value{, specifically,} 
 how their values change as the uncertainty-level parameter $\tau$ increases   
or {approaches} infinity.

\subsection{The case of uncertainties in the constraints}
We consider the following problem 
and assume that {Assumption~\ref{assum_cons}} holds:
\begin{equation}
     \label{prob:P_a_cons}
     \tag{$\text{P}_{a\text{ cons}}$}
     \begin{aligned}
          &{\displaystyle \min_x} \quad f_0(x) \\
          & \ {\rm s.t.} \hspace{3mm} 
               f(a, x) \geq 0 {,}  \\
          & \hspace{9mm} 
               x \in S.
     \end{aligned}
\end{equation}
\begin{assum} \label{assum_cons}
     We assume that the problem \eqref{prob:P_a_cons} satisfies the following conditions:
     \begin{enumerate}
     \item The set $S$ is {convex}.
     \item The uncertain parameter $a$ {takes} only positive values, 
           {i.e.}, $a \in \R {^m} _{++}$.
     \item The function $f \colon \R^m \times \Rn \rightarrow \R$ is convex with respect to~$a$ and concave with respect to~$x$.
     \item The function $f$ is differentiable in $a$ over $\R^m_{++}$.
     \end{enumerate}
\end{assum}
The robust counterpart of the problem \eqref{prob:P_a_cons} 
is given {by}
\begin{equation*}
     \label{prob:RC_P_cons}
     \tag{$\text{RC}_{\text{P cons}}$}
     \begin{aligned}
          &{\displaystyle \min_x} \quad f_0( x) \\
          & \ {\rm s.t.} \hspace{3mm} 
               \min_{a \in \Omega(a^0, A, \tau)} f(a, x) \geq 0 {,}  \\
          & \hspace{9mm} 
               x \in S.
     \end{aligned}
\end{equation*}
Throughout this subsection,
we {vary} only the value of $\tau$ 
{while} fixing in advance the parameters $a^0$ and $A$ 
{in} the set $\Omega(a^0, A, \tau)$.
Hence, for simplicity, we denote $\Omega(a^0, A, \tau)$ by~$\Omega(\tau)$.

\begin{remark}
    In Sections \ref{sect:pre} and \ref{sect:log}, we discussed the standard robust constraint of the form
$\max_{a\in U} f(a,x)\le 0$, which is convenient for presenting the Fenchel-duality-based
reformulation. In this section, however, we consider the equivalent form
\[
\min_{a\in\Omega(\tau)} f(a,x)\ge 0,
\]
because the effect of positivity preservation is more naturally interpreted through the
worst-case scenario when $f(\cdot,x)$ is increasing. These two formulations are equivalent
up to a sign change, but the latter is more suitable for the analysis in this section.
\end{remark} 

First, we define the estimated value of the uncertain parameter~$a$ 
in the worst case.  
%
%
\begin{dfn}
     In the robust optimization problem~\eqref{prob:P_a_cons}, 
     we call the vector 
     \[
          a^*(\tau; x) \coloneqq \argmin_{a \in \Omega(\tau)} f(a, x)     
     \]
     the worst-case scenario of the parameter $a$ at $x$.
     In addition, we call the worst-case scenario of $a$ 
     at the robust optimal solution $x^*(\tau)$ of \eqref{prob:RC_P_cons}
     \[
          a^{**} (\tau) \coloneqq a^*(\tau; x^*(\tau)) = 
          \argmin_{a \in \Omega(\tau)} f(a, x^*(\tau))      
     \]
     the robust scenario of $a$.
\end{dfn}
\noindent Note that the robust constraint in~\eqref{prob:P_a_cons} 
can be written via the worst-case scenario~as
\begin{equation*}
     f(a^*(\tau; x), x) \geq 0, 
\end{equation*}
which means that the worst-case scenario gives the risk-adjusted constraint.

The following {lemma} is {obtained by} 
evaluating the worst-case scenario for {each} value of~$\tau$.
It characterizes the worst-case scenario defined above.
\begin{lemma} \label{lemma_z^*}
     Assume that the robust optimization problem~\eqref{prob:P_a_cons} satisfies
     Assumption~\ref{assum_cons} 
     and {that} the function {$f(\cdot, x)$} is either increasing or decreasing.
     In addition, we define the set 
     $\Omega'(a^0, \tau, A) \subset \R^m \times \R^m \times \R^m$ as 
     \begin{equation*}
          \Omega'(a^0, \tau, A) \coloneqq 
               \left\{ 
                    (a, z, y) \in \R^m \times \R^m \times \R^m \bigmid
                    a = A^0 z, \ 
                    z_i - \ln z_i - 1 \leq y_i, \
                    \| A y \| \leq \tau
               \right\},
     \end{equation*}
     and for $x \in S${,} let {$\mathcal{I}(x)$} be the set of indices defined by
     $$\mathcal{I} (x) \coloneqq \left\{ i \bigmid 
                         \frac{\partial}{\partial a_i}f(a,x) 
                         \neq 0, \ 
                         \forall a \in \R^m_{++} \right\}.$$
     Then,
     the worst-case scenario $a^*(\tau; x)$ at the point $x$ 
     and the vectors $z^*, y^* \in \R^m$ 
     such that $(a^*(\tau; x), z^*, y^*) \in \Omega'(a^0, \tau, A)$
     satisfies the following statements for $i \in \mathcal{I}(x)$:
     \begin{align}
          \begin{aligned} \label{eqs_lemma_wc_s_incr}
               a^*(\tau; x) &= A^0 z^*, \\
               g(z^*_i) &= y^*_i, \quad i \in \mathcal{I}(x), \\
               \| A y^* \| &= \tau. 
          \end{aligned}
     \end{align}
     Moreover, the following statements also hold:
     \begin{itemize}
          \item If {$f(\cdot, x)$} is increasing, 
               then $0 < z^*_i < 1$ 
               for $i \in \mathcal{I}(x)$;
          \item If {$f(\cdot, x)$} is decreasing, 
               then $z^*_i > 1$ 
               for $i \in \mathcal{I}(x)$.
     \end{itemize}
\end{lemma}
\begin{proof}
     It suffices to consider the case
     where $f(\cdot, x)$ is increasing.
     The minimization problem
     $\min_{{a \in \Omega(\tau)}} f(a, x)$ {can be} written as 
     \begin{equation*}
          \begin{aligned}
               & \min_{a, z, y} f(a, x) \\
               & \ {\rm s.t.} \hspace{3mm} 
                    a - A^0 z = 0, \\
               & \hspace{9mm} 
                    z_i - \ln z_i - 1 \leq y_i, \quad i = 1, \dots, m, \\
               & \hspace{9mm} 
               \| A y \| \leq \tau, 
          \end{aligned}
     \end{equation*}
     and we obtain its KKT {conditions} as follows:
     \begin{gather}
          \nabla_a f(a^*, x) 
               + \theta  = 0, \label{KKT_a} \\
          -\theta_i a^0_i 
               + \lambda_i \left(1 - \frac{1}{z^*_i}\right) = 0, \quad 
               i = 1, \dots, m,
               \label{KKT_z} \\        
          - \lambda 
               + \lambda_{m+1}  \eta^* = 0
               \label{KKT_y} \\
          z^*_i - \ln z^*_i -1 \leq y^*_i, \quad i = 1, \dots, m, \quad 
               \| A y^* \| \leq \tau, \label{KKT_lag} \\
          \lambda_i ( z^*_i - \ln z^*_i -1 - y^*_i) = 0, \quad i = 1, \dots, m,
               \label{KKT_com_m} \\
          \lambda_{m+1} ( \| Ay^* \| - \tau ) = 0, \label{KKT_com_m+1} \\
          \lambda_i \geq 0,  \quad i = 1, \dots, m + 1 \notag,
     \end{gather}
     where 
     $\eta^* \in \partial \| A y^* \| $.
     By multiplying $A^0$ from the left to (\ref{KKT_a}),
     we get
     \[
          A^0 \nabla_a f(a^*, x) 
          + \big[\theta_i a^0_i \big]_{i=1}^m = 0.
     \]
     Substituting~\eqref{KKT_z} into the above equation, {we obtain}
     \begin{equation}
          a^0_i \frac{\partial}{\partial a_i} f(a^*, x) 
          + \lambda_i \left(1 - \frac{1}{z^*_i}\right) = 0, \quad 
          i = 1, \dots, m.
          \label{KKT_a_z}
     \end{equation}
     Now, note that $\lambda_i > 0$ when $i \in \mathcal{I}(x)$ since $a_0 \in \R^m_{++}$ and $\frac{\partial}{\partial a_i} f(a^*, x) > 0$ in this case. Then, the above equality yields
     \begin{equation} \label{z^*_i < 1}
          z^*_i = 
          \frac{\lambda_i}
          {\lambda_i + a^0_i \frac{\partial}{\partial a_i} f(a^*, x)}, \quad
          i \in \mathcal{I}(x).
     \end{equation}
     Therefore we see from (\ref{z^*_i < 1}) that $z^*_i < 1$
     for $i \in \mathcal{I}(x)$.
     {Next} we see that $\lambda_{m + 1} \neq 0$.
     Suppose{,} for {the sake of} contradiction{,} that $\lambda_{m + 1} = 0$.
     Then from (\ref{KKT_y}) we have 
     $\lambda_i = 0, \ i = 1, \dots, m$,
     and then it follows from \eqref{KKT_z} and \eqref{KKT_y} that 
     $\frac{\partial}{\partial a_i} f(a^*, x) = 0$,
     which contradicts the assumption that $i \in \mathcal{I}(x)$.
     Thus we obtain $\lambda_{m + 1} \neq 0$ and then 
     $\| A y^* \| = \tau$ from (\ref{KKT_com_m+1}).
     Since $\lambda_{m + 1} \neq 0$, 
     if $\frac{\partial}{\partial a_i} f(a^*, x) \neq 0$
     then from (\ref{KKT_a_z}), we have 
     $\lambda_i \neq 0$.
     So we obtain from (\ref{KKT_com_m}) {that}
     \[
          z^*_i - \ln z^*_i -1 = y^*_i, \quad
          i \in \mathcal{I}(x).
     \]
     In the case where $f(\cdot, x)$ is decreasing,
     we have $\frac{\partial}{\partial a_i} f(a^*, x) < 0$ 
     for $i \in \mathcal{I}(x)$.
     Therefore, the statement can be verified similarly.
\end{proof}
This lemma {implies} that if the constraint function {$f(\cdot, x)$} is {either} increasing or decreasing, 
then the worst-case scenario is located at the {boundary} of the uncertainty set \eqref{eq:omega}, 
as {illustrated in} Figure \ref{fig:wc_scenario}.

\begin{figure}[ht] 
     \centering
     \includegraphics[width=5cm]{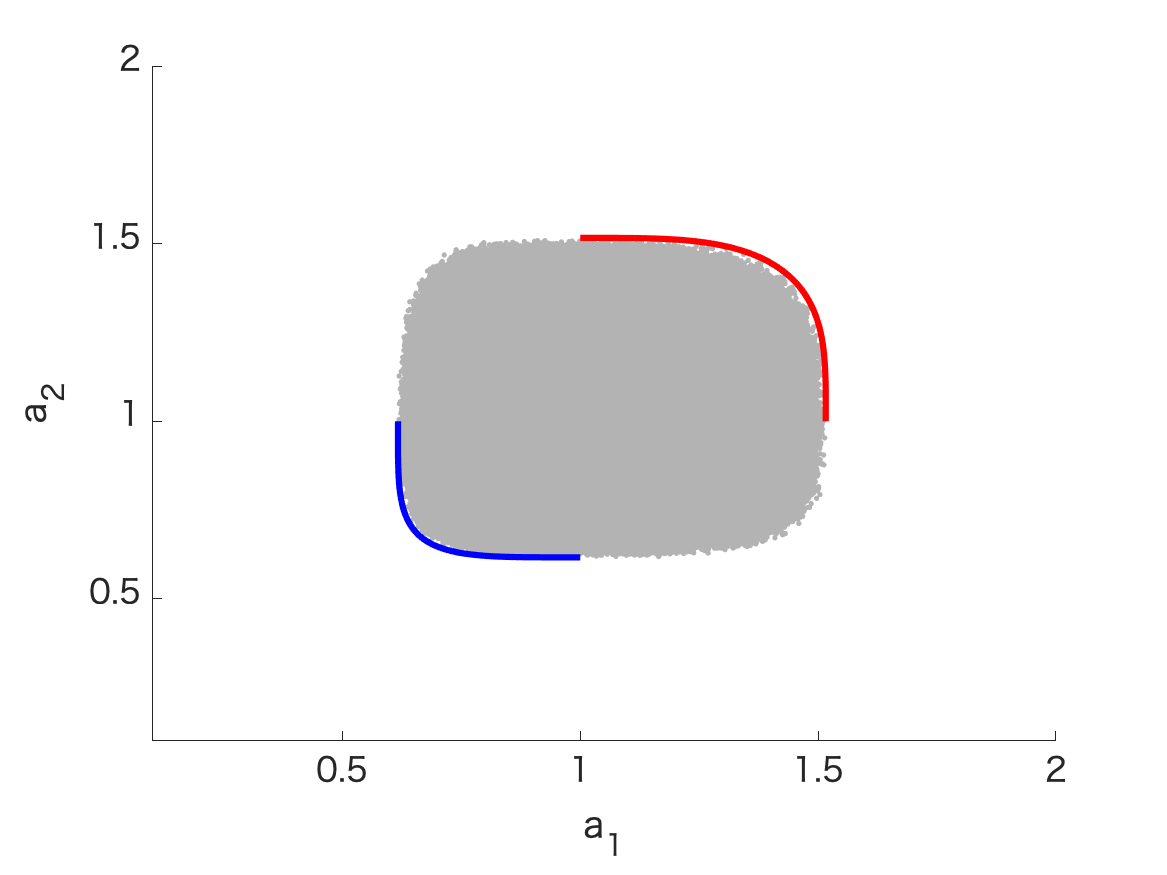}
     \caption{Description of Lemma~\ref{lemma_z^*}. 
     The blue and red lines denote
     the cases where $f$ is increasing and decreasing, 
     respectively.}
     \label{fig:wc_scenario}
\end{figure}

%
%
\begin{thm} \label{thm:a_bound}
     Assume that the robust optimization problem~\eqref{prob:P_a_cons} satisfies
     Assumption~\ref{assum_cons} 
     and the function {$f(\cdot, x)$} is either increasing or decreasing.
     In addition, we define the functions $g_+$, and $g_-$ as 
     \begin{align*}
          g_-(t) &\coloneqq t - \ln t - 1, \quad 0 < t \leq 1, \\
          g_+(t) &\coloneqq t - \ln t - 1, \quad 1 \leq t < \infty.
     \end{align*}
     Then,
     the worst-case scenario $a^*(\tau; x)$ at the point $x$ satisfies 
     the following statements:
     \begin{itemize}
          \item If {$f(\cdot, x)$} is increasing, 
               then for $x \in S$ that satisfies $\nabla_a f(a,x)\neq 0$
               we have
               \begin{equation} \label{ineq:th_a^*_bound_incr}
                    g_-^{-1}(\tau \|A^{-1}\|) \ a^0_i \leq a^*_i(\tau; x) \leq a^0_i, 
                    \quad i \in \mathcal{I}(x).   
               \end{equation} 
          \item If {$f(\cdot, x)$} is decreasing, 
               then for $x \in S$ that satisfies $\nabla_a f(a,x)\neq 0$
               we have
               \begin{equation} \label{ineq:th_a^*_bound_decr}
                    a^0_i \leq a^*_i(\tau; x) \leq g_+^{-1}(\tau \|A^{-1}\|) \ a^0_i,
                    \quad i \in \mathcal{I}(x).   
               \end{equation}
     \end{itemize}
\end{thm}
\begin{proof}
     Let us consider the increasing case. 
     By Lemma \ref{lemma_z^*}, $a^*(\tau; x)$, $z^*$, and $y^*$ 
     satisfy~\eqref{eqs_lemma_wc_s_incr}.
     Then we have
     \[ 
          \| y^* \| \leq \| A^{-1} \| \| A y^* \| 
          = \tau \| A^{-1} \|, 
     \]
     which implies 
     \begin{equation} \label{lemma_proof_y_bound}
          y^*_i \leq \tau \| A^{-1} \|, 
          \quad i = 1, \dots, m.
     \end{equation}
     From \eqref{eqs_lemma_wc_s_incr} 
     it also follows that
     \begin{equation} \label{lemma_proof_z=g^-1(y)}
          z^*_i = g_-^{-1}(y^*_i), \quad i \in \mathcal{I}(x).
     \end{equation}
     Note that $g_-$ is a {decreasing} function, and thus 
     from (\ref{lemma_proof_y_bound}) and (\ref{lemma_proof_z=g^-1(y)}) we obtain
     \begin{equation*}
          g_-^{-1}(\tau \| A^{-1} \|) \leq z^*_i \leq 1, \quad
          i \in \mathcal{I}(x), 
     \end{equation*}
     which is equivalent to (\ref{ineq:th_a^*_bound_incr}).
     In the case {where} $f(\cdot, x)$ is decreasing, we can  verify
     the statement {in a similar way}.
     By Lemma \ref{lemma_z^*} we have (\ref{lemma_proof_y_bound}) and
     $z^*_i = g_+^{-1}(y^*_i)$ for $i \in \mathcal{I}(x)$.
     Therefore we obtain
     \begin{equation*}
          1 \leq z^*_i \leq g_+^{-1}(\tau \| A^{-1} \|), \quad
          i \in \mathcal{I}(x), 
     \end{equation*}
     and then (\ref{ineq:th_a^*_bound_decr}) follows.
\end{proof}
%
%
%
Theorem \ref{thm:a_bound} shows the following three advantages of our uncertainty model.
\begin{enumerate}
     \item We {provide} a bound {for} the solution of the problem 
          $\min_a \{f(a, x) \mid a \in \Omega(\tau)\}$,
          which can be explicitly calculated with $a^0$, $\tau$, and $A$.
          In the case of a general uncertainty set or a constraint 
          function~$f$, we may not solve the minimization explicitly.
     \item This is a result that supports the ``positivity'' of the worst-case scenario 
          for arbitrary $\tau$ if we use the uncertainty set \eqref{eq:omega}.
          {In particular}, the difference from existing models is observed 
          when $f$ is increasing for $a$. 
          In this case, the lower bound of $a^*(\tau; x)$ converges to zero as $\tau \to \infty$, 
          which implies that $a^*(\tau; x) > 0$ for arbitrary $\tau$.
     \item We can use this result {to determine} the level of assumed uncertainty $\tau$.
          We present the details in Subsection \ref{subsect:log_params}.
\end{enumerate}

%
%
The following proposition shows the limit of the worst-case scenario 
as $\tau \to \infty$. 
\begin{thm} \label{thm_a_lim}
     Assume that the robust optimization problem~\eqref{prob:P_a_cons} satisfies
     Assumption~\ref{assum_cons} 
     and the function {$f(\cdot, x)$} is either increasing or decreasing.
     Then,
     the worst-case scenario $a^*(\tau; x)$ at the point $x$ satisfies 
     the following statements:
     \begin{itemize}
          \item If {$f(\cdot, x)$} is increasing, 
               then for $x \in S$ that satisfies $\nabla_a f(a,x)\neq 0$ we have
               \begin{equation} \label{eq:thm_lim_a}
                    \lim_{\tau \to \infty} a^*_i(\tau; x) = 0,
                    \quad i \in \mathcal{I}(x).   
               \end{equation}
          \item If {$f(\cdot, x)$} is decreasing, 
               then for $x \in S$ that satisfies $\nabla_a f(a,x)\neq 0$ we have
               \begin{equation} \label{eq:thm_lim_a_decreasing}
                   \lim_{\tau \to \infty} a^*_i(\tau; x) = \infty,
                    \quad i \in \mathcal{I}(x). 
               \end{equation}
     \end{itemize}
\end{thm}
\begin{proof}
     Let us consider the increasing case. 
     By Lemma \ref{lemma_z^*}, $a^*(\tau; x)$, $z^*$, and $y^*$ satisfy~\eqref{eqs_lemma_wc_s_incr}.
     Now, let $Y(\tau)$ be the set defined by
     \begin{equation*}
          Y(\tau) = 
          \left\{ 
               {y \in \R^m_{++}} \bigmid \| A y \| = \tau
          \right\},
     \end{equation*}
     and let us verify that 
     \begin{equation} \label{eq:proof_Y}
          Y(\alpha \tau) = \alpha Y(\tau), \quad 
          \forall \tau, \alpha > 0.
     \end{equation}
     If $y \in Y( \alpha \tau)$, then $y$ satisfies
     $\| A y \| = \alpha \tau$.
     Hence we have $\| A (1 / \alpha) y \| = \tau$,
     which implies $(1 / \alpha) y \in Y(\tau)$.
     Thus it follows that $y \in \alpha Y(\tau)$,
     and we obtain $Y(\alpha \tau) \subseteq \alpha Y(\tau)$.
     On the contrary, let $y \in \alpha Y(\tau)$. 
     Then $y$ satisfies $\| A (1 / \alpha) y \| = \alpha \tau$,
     which implies $\| A y \| = \alpha \tau$.
     Therefore we also have $Y(\alpha \tau) \supseteq \alpha Y(\tau)$,
     and then (\ref{eq:proof_Y}) follows.
     \par
     By letting $\alpha = {\tau_0 / \tau}$ in (\ref{eq:proof_Y}), we have
     \begin{equation*}
          Y(\tau) = \frac{\tau}{\tau_0} Y(\tau_0), \quad
          \forall \tau, \tau_0 > 0,
     \end{equation*}
     and then $Y(\tau)$ is expressed as
     \begin{equation*}
          Y(\tau) = 
          \left\{ 
               y \in \R^m_{+} \bigmid y = \frac{\tau}{\tau_0} y_0, \
               \| A y_0 \| = \tau_0, \ {y_0 > 0 }
          \right\}.
     \end{equation*}
     {Note that $y^* > 0$ from Lemma~\ref{lemma_z^*}, then we have $y^* \in Y(\tau)$.}
     Thus we obtain 
     \begin{equation*}
          y^*_i = \frac{\tau}{\tau_0} (y_0)_i, \
          {(y_0)_i > 0}
     \end{equation*}
     for $i \in \mathcal{I}(x)$, which implies
     \begin{equation} \label{eq:proof_y_lim}
          \lim_{\tau \to \infty} y^*_i(\tau; x) = \infty,
                    \quad i \in \mathcal{I}(x).   
     \end{equation}
    In the increasing case, we observe that $z^*$ goes to 0 from the second equality of~\eqref{eqs_lemma_wc_s_incr} and~\eqref{eq:proof_y_lim}.
Thus we have~\eqref{eq:thm_lim_a} from the first equality of~\eqref{eqs_lemma_wc_s_incr}.
In the case where $f$ is decreasing, $z^*$ approaches to $\infty$,
which results in~\eqref{eq:thm_lim_a_decreasing} by a similar argument.
\end{proof}

\begin{remark}
     In the case where $f(\cdot, x)$ is increasing, 
     we can evaluate the convergence speed 
     of the bound of the worst-case scenario given in~\eqref{ineq:th_a^*_bound_incr}.
     We have the following inequality:
     \begin{equation*}
          \beta (\alpha, \delta) \ 
          {\rm exp} \left(- \frac{y}{\alpha} \right) \leq 
          g^{-1} (y) \leq 
          {\rm exp} (- y), \quad 
          y \geq 0,
     \end{equation*} 
     where $\beta(\alpha, \delta) = \exp(- \delta / \alpha)$ 
     and $\alpha, \delta \in (0, 1)$ satisfy 
     $\delta \geq - (1 - \alpha) + (1 - \alpha) \ln (1 - \alpha) + 1$.
     Then if we set for example $\alpha = 1 / 2$ 
     then we can evaluate the bound in~\eqref{ineq:th_a^*_bound_incr}:
     \begin{equation*}
          \frac{2 a^0_i}{e} \ 
          {\rm exp} \left(- 2 \|A^{-1}\| \tau \right) \leq 
          g_-^{-1} \left(\|A^{-1}\| \tau \right) \ a^0_i \leq 
          a^0_i {\rm exp} \left(- \|A^{-1}\| \tau \right).
     \end{equation*}
\end{remark}
%
%

The following proposition shows that the robust optimal value is increasing with respect to~$\tau$, which indicates that the robust optimal value becomes worse, in the sense of the original problem~\eqref{prob:P_a_cons}, as the level of assumed uncertainty increases.
\begin{prop} \label{prop:opt_value_decr}
     Let $x^*(\tau_1)$ and $x^*(\tau_2)$ be robust optimal {solutions} 
     of the robust counterpart \eqref{prob:RC_P_cons} for 
     $\tau = \tau_1, \tau_2$, respectively.
     Then 
     $$
          \tau_1 \leq \tau_2 \implies 
          {f_0(x^*(\tau_1)) \leq  f_0(x^*(\tau_2))}.
     $$
\end{prop}
\begin{proof}
     This can be verified from Proposition \ref{prop:property_basic}.
     Let $X(\tau)$ be the feasible set of the robust counterpart~\eqref{prob:RC_P_cons}.
     Then the proposition shows that 
     $\tau_1 \leq \tau_2$ implies $X(\tau_1)\supseteq X(\tau_2)$.
     Thus we have {$f_0(x^*(\tau_1)) \leq f_0(x^*(\tau_2))$}.
\end{proof}
As in Proposition~\ref{prop:property_basic}, the above result also implies that the parameter~$\tau$ corresponds to the level of uncertainty assumed in our model.


\subsection{The case of uncertainties in the objective function}
Next, we consider the case of uncertainty 
in the objective function, formulated as follows:
\begin{equation*}
     \label{prob:P_a_obj}
     \tag{$\text{P}_{a\text{ obj}}$}
     \begin{aligned}
          &{\displaystyle \max_x} \quad f(a, x) \\
          & \ {\rm s.t.} \hspace{5mm} x \in S.
     \end{aligned}
\end{equation*}
\begin{assum} \label{assum_obj}
     Assume that the problem~\eqref{prob:P_a_obj} satisfies 
     the following conditions.
     \begin{enumerate}
     \item The set $S$ is convex.
     \item The uncertain parameter $a$ {takes} only positive values, 
          {i.e.}, $a \in {\R^m_{++}}$.
     \item The function $f \colon \R^m \times \Rn \rightarrow \R$ is convex with respect to~$a$ 
          and concave with respect to~$x$.
     \end{enumerate}
\end{assum}
The robust counterpart of the problem~\eqref{prob:P_a_obj} 
obtained by using the proposed uncertainty set is written as follows:
\begin{equation}
     \label{prob:RC_P_obj2}
     \tag{$\text{RC}_{\text{P obj}}$}
     \begin{aligned}
          & \hspace{-3mm} {\displaystyle \max_{x}} \quad 
               \min_{a \in \Omega(a^0, A, \tau)} f(a,x) \\
          & \hspace{-1mm} {\rm s.t.} \hspace{7mm} x \in S.
     \end{aligned} 
\end{equation}

\begin{remark}
     The robust counterpart~\eqref{prob:RC_P_obj2} 
     contains uncertainty in the objective function. 
     Note that it is essentially equivalent to uncertainty in a constraint.
     Indeed, if we introduce new auxiliary variables $t \in \R$, 
     then the problem~\eqref{prob:RC_P_obj2} can be rewritten as
     \[ 
     \begin{aligned}
          &{\displaystyle \max_{x, t}} \quad  t \\
          & \ {\rm s.t.} \hspace{5mm} \min_{a \in \Omega(a^0, A, \tau)} \
               f(a, x) - t \geq 0 \\
          & \hspace{10mm} x \in S. \\
     \end{aligned}
     \]
     This results in a problem with uncertain parameters in its constraint. 
\end{remark}

For the problem~\eqref{prob:RC_P_obj2}, 
as in the previous case, 
the boundedness and limit of the worst-case scenario, 
and the monotonicity of robust optimal value also hold.
However, we present the expression and proof for
the monotonicity of robust optimal value again,
as it differs from the previous case.
\begin{prop}
     Let $x^*(\tau_1)$ and $x^*(\tau_2)$ be the robust optimal solution 
     of the robust counterpart \eqref{prob:RC_P_obj2} for 
     $\tau = \tau_1, \tau_2$, respectively.
     Then 
     $$
          \tau_1 \leq \tau_2 \implies 
          {f(a^{**}(\tau_1), x^*(\tau_1))} \geq {f(a^{**}(\tau_2), x^*(\tau_2))}.
     $$
\end{prop}
\begin{proof}
     Since $\tau_1 \leq \tau_2$, it follows from Proposition \ref{prop:property_basic} that 
     $\Omega(\tau_1) \subseteq \Omega(\tau_2)$.
     This implies that for any $x \in S$, the minimum over the larger set cannot be greater than the minimum over the smaller set. Thus, we have:
     \begin{equation*}
          \min_{a \in \Omega(\tau_1)} f(a, x) \geq \min_{a \in \Omega(\tau_2)} f(a, x), \quad \forall x \in S.
     \end{equation*}
     Taking the maximum over all $x \in S$ on both sides preserves the inequality, yielding:
     \begin{equation*}
          \max_{x \in S} \min_{a \in \Omega(\tau_1)} f(a, x) \geq \max_{x \in S} \min_{a \in \Omega(\tau_2)} f(a, x).
     \end{equation*}
     By the definition of the robust optimal value, this is exactly 
     $$f(a^{**}(\tau_1), x^*(\tau_1)) \geq f(a^{**}(\tau_2), x^*(\tau_2)),$$
     which completes the proof.
\end{proof}
\par
The following theorem shows that 
we can evaluate the robust optimal value for every value of~$\tau$.
This is clearly verified by duality.
\begin{thm} \label{thm:opt_bound}
     Assume that the robust optimization problem~\eqref{prob:P_a_obj} satisfies
     Assumption~\ref{assum_obj}.
     Then,
     the robust optimal value 
     $f^*_{\rm robust}(\tau) \coloneqq \max_x \min_a f(a, x)$ 
     satisfies the following statements{:}
     \begin{equation*}
          m(\tau; x, u, v, s^1) \leq f^*_{\rm robust}(\tau), 
           \quad \forall (x, u, v, s^1) \in \mathcal{F}, 
     \end{equation*}
     where the function $m: \Rn \times \R^m \times \R^m \times \R^m \to \R$
     and the set $\mathcal{F} \subset \Rn \times \R^m \times \R^m \times \R^m$
     are defined, respectively, as
     \begin{align*}
          m(\tau; x, u, v, s^1) =
               \sum_{k = 1}^m u_k 
               \ln \left( 1 - \frac{a^0_k v_k}{u_k} \right) 
               - \tau \| s^1 \|_* 
               - f^*(-v, x), \\
          \mathcal{F} = 
               \left\{ (x, u, v, s^1) \in \Rn \times \R^m \times \R^m \times \R^m \mid
               x \in S, \ u = A^{\T} s^1, \ u \in \R^m _{++}, \
               u - A^0 v \in \R^m _{++} \right\}.
     \end{align*}
\end{thm}
We can obtain one lower bound if we choose one set of  $(x, u, v, s^1) \in \mathcal{F}$.
An example is presented below.
\begin{example} \label{ex:opt_bound}
     If $f(a, x) = a^{\T} x$, lower bounds $m$ are given as
     \begin{gather*}
          m(\tau; x, u, s^1) =
               \sum_{k = 1}^n u_k 
               \ln \left( 1 + \frac{a^0_k x_k}{u_k} \right) 
               - \tau \| s^1 \|_*, \\
          \mathcal{F} = 
               \left\{ (x, u, s^1) \in \Rn \times \R^m \times \R^m \mid
               x \in S, \ u = A^{\T} s^1, \ u \in \R^m _{++}, \
               u + {A^0} x \in \R^m _{++} \right\}.
     \end{gather*}
     Then, in this case, if we set 
     $u = t A^0 x^0 \ (t > 0)$, 
     $s^1 = (A^{\T})^{-1} u$ for $x^0 \in S$, we have a lower bound
     \begin{equation} \label{eq:ex_opt_bound}
          m(\tau; x^0, u, s^1) = 
          t \ln \left(1 + \frac{1}{t} \right) (a^0)^{\T} x^0
          - t \tau \| (A^{\T})^{-1} A^0 x^0 \|_*.
     \end{equation}
\end{example}
\begin{cor}
     Assume that the robust optimization problem~\eqref{prob:P_a_obj} satisfies
     Assumption~\ref{assum_obj} 
     and the function {$f(\cdot, x)$} is increasing.
     Then,
     the following statements hold for $x \in S$ that satisfy 
     $\nabla_a f(a, x) \neq 0$:
     \begin{itemize}
          \item $f(a^*(\tau; x), x) \geq
               {\displaystyle \lim_{a \downarrow 0}} \ f(a, x)$ 
               for all $\tau >0$.
          \item ${\displaystyle \lim_{\tau \to \infty}} \
               f(a^*(\tau; x), x) = 
               {\displaystyle \lim_{a \downarrow 0}} \ f(a, x)$.
     \end{itemize}
\end{cor}
\begin{proof}
     The first statement is easily verified 
     by the assumption that $f(\cdot, x)$ is increasing.
     The second statement is clear from Theorem~\ref{thm_a_lim}
     since $f(\cdot, x)$ is a continuous function.
\end{proof}


\subsection{Determination of parameters in the proposed model} \label{subsect:log_params}
As we have seen before, 
the proposed uncertainty model \eqref{eq:omega} has three parameters 
$a^0$, $\tau$, and $A$.
In this subsection, we introduce {methods} to determine their values
when we are given data of the uncertain parameter $a \in \R^m_{++}$.
Throughout this subsection, let the vectors $a^1, \dots, a^N$ 
denote $N$ observed data vectors.
\subsubsection*{Nominal value $a^0$}
We can use the sample mean of the given data
$\bar{a} = (1 / N) \sum_{i = 1}^N a^i$.
%
%
\subsubsection*{Matrix $A$}
We determine the matrix $A$, which captures the correlation between the components $a_i$ and $a_j$,
based on the covariance matrix computed from the data.
The concrete procedure is as follows:
\begin{enumerate}
     \item Let $A^0 \coloneqq \text{diag} (a^0)$ and 
           transform the data $a^i$ using {$A^0$} as  
           $z^i = (A^0)^{-1} a^i, \ i = 1, \dots, N$.
     \item Measure the variation of the data $z^i$ from {$e$} 
           with the function $g(t) = t - \ln t - 1$ as 
           $y^i_j = g(z^i_j), \ j = 1, \dots, m$.
     \item Compute the covariance matrix $\Sigma$ of 
           the vectors $y^i, \ i = 1, \dots, N$.
     \item Set $A = \Sigma^{- 1 / 2}$.
\end{enumerate}
This procedure is analogous to the one used to determine the parameters of an
ellipsoidal uncertainty set.
If the number of data points is sufficiently large, 
then the covariance matrix $\Sigma$ is symmetric positive definite, and thus $A = \Sigma^{-1/2}$ is also symmetric positive definite.
%
%
\subsubsection*{The level of assumed uncertainty $\tau$}
The value of the parameter $\tau$ is especially important 
because it corresponds to the size of our uncertainty model, in other words, 
the level of uncertainty assumed by the model.
We introduce three ways to choose the value of $\tau$.
Note that {in what follows} we assume that the values of the other parameters $a^0$ and $A$ {are} 
already {fixed}. 

%
The first {method is to choose the smallest value of $\tau$ such that} all the observed data $a^1, \dots, a^N$ 
belong to the uncertainty set. 
The following proposition is useful for determining whether
a given point $\hat{a}$ belongs to the uncertainty set.
\begin{prop} \label{prop:a_in_Omega}
     Assume that a vector $\hat{a} \in \R^m_{++}$ satisfies
     \begin{equation*}
          \| A \hat{y} \| \leq \tau,
     \end{equation*}
     where
     $\hat{y}_j = g(\hat{a}_j / a^0_j), \ j = 1, \dots, m$.
     Then $\hat{a}$ belongs to the set $\Omega(a^0, \tau, A)$ 
     defined in \eqref{eq:omega}.
\end{prop}
\begin{proof}
     Let $\hat{z} = (A^0)^{-1} \hat{a}$ 
     and then we have $\hat{y}{_j} = g(\hat{z}{_j})$.
     Hence, by the definition of $\Omega(a^0, \tau, A)${,}
     it is clear that $\hat{a}$ belongs to $\Omega(a^0, \tau, A)$.
\end{proof}
From Proposition \ref{prop:a_in_Omega}, 
all the observed data belong to the uncertainty set if we set
\begin{equation*}
     \tau_{\rm full} \coloneqq \max_{i = 1, \dots, N} \| A y^i \|,
\end{equation*}
where
$y^i_j = g ( a^i_j / a^0_j )$,  
$i = 1, \dots, N, \ j = 1, \dots, m$.
%

%
The second method uses the lower bound of the worst-case scenario
obtained in Theorem~\ref{thm:a_bound}.
Consider the case where $f(\cdot, x)$ is increasing.
If we want to ensure that the $i$th element of the worst-case scenario 
$a^*(\tau, x^*(\tau))$ is larger than $\beta_i a^0_i$, then we can set~$\tau$ as 
\begin{equation*}
     \tau \leq \min_{i = 1, \dots, {m}} \frac{g(\beta_i)}{\| A^{-1} \|}.
\end{equation*}
%

%
The third method applies to the case 
where the objective function of the optimization problem contains uncertain parameters.
In this case,  we can utilize the lower bound of the robust optimal value
obtained in Theorem~\ref{thm:opt_bound}.
For example, if we want to make the robust optimal value of the problem 
considered in Example~\ref{ex:opt_bound} larger than~$\gamma$, 
then, by~\eqref{eq:ex_opt_bound}, it is sufficient to choose $\tau$ such that
\begin{equation*}
     \tau \leq 
     \frac{t \ln \left(1 + \frac{1}{t} \right) (a^0)^{\T} x^0 - \gamma}
     {t \| (A^{\T})^{-1} A^0 x^0 \|_*}.
\end{equation*}


\subsection{Probabilistic guarantee}

In this subsection, we derive a probabilistic guarantee result 
for the feasibility of the robust feasible solutions 
obtained by the proposed model.
Let us consider the following constraint in the problem~\eqref{prob:P_a_cons}:
\[
    f(a, x) \geq 0,
\]
and the robust counterpart is given as
\begin{equation} \label{ineq:rc_pro}
    \min_{a \in \Omega(a^0, A, \tau)} f(a, x) \geq 0.
\end{equation}
Note that the following result does not need any assumptions on the constraint function~$f$ such as convexity or monotonicity. 

\begin{assum} \label{assum:pro_gua}
    We assume that
    the uncertain parameter $a \in {\R^m _{++}} $ follows 
    the lognormal distribution 
    with mean $\mu$ and covariance matrix $\Sigma$, 
    and that each element of $a$ is independent.
\end{assum}

The lognormal distribution is one of the most common probability distributions
with only positive values.
The assumption of the independence of each $a_i$ implies 
that $\Sigma$ is a diagonal matrix 
and that the random variables $\ln a_i$ are also independent. 
Note also that, from \cite{Ait57}, $\mu^{\ln}$ and $\Sigma^{\ln}$ are written via $\mu$ and $\Sigma$ as 
\begin{equation*}
    \mu^{\ln}_i = \ln \mu_i - \frac{1}{2} \ln \left(1 + \frac{\sigma_i^2}{\mu_i^2} \right), \quad
    (\sigma^{\ln}_i)^2 = \ln \left(1 + \frac{\sigma_i^2}{\mu_i^2} \right)
\end{equation*}
{where $\sigma^2_i \coloneqq \Sigma_{i,i}$.}

Under this assumption, 
let $\mu^{\ln}$ and $\Sigma^{\ln}$ be the mean and the covariance matrix 
of the normal distribution to which the random vector $\ln a$ follows, respectively,  
where $\ln a = (\ln a_1, \dots, \ln a{_m})^{\T}$.
Moreover, let $\lambda$ be the maximum eigenvalue of the diagonal matrix $\Sigma^{\ln}${, i.e., $\lambda \coloneqq \max_i \lambda_i$}.
We first show the following lemma that will be used in the main Theorem~\ref{thm:pro_gua}.
\begin{lemma} \label{lemma:g_ln}
    Let $\alpha > 0$. 
    Then the following inequality holds: 
    \begin{equation*}
        \alpha g(t) \leq \left| \ln t \right|, \quad 
            \forall t \in 
            {\left( \max \left\{ 0, 1 - \frac{1}{\alpha} \right\}, 
            1 + \frac{1}{\alpha} \right]}.
    \end{equation*} 
\end{lemma}

\begin{proof}
    First, we consider the case of $0 < t < 1$.
    When $\alpha \in {(}0, 1]$, 
    from $g(t) \geq 0$ we have 
    $- \ln t - \alpha g(t) \geq - \ln t - g(t) = 1 - t > 0$.
    In the case of $\alpha \geq 1$, 
    {define} the function $h_{\alpha} (t)$ as $h_{\alpha} (t) \coloneqq | \ln t | - \alpha g(t)$.
    Then we have $h'_{\alpha} (t) =  ((\alpha - 1) - \alpha t) / t$, 
    and this derivative is non-positive  
    for all $t \in \left[ 1 - \frac{1}{\alpha}, 1 \right)$.
    Note that $h_{\alpha} (1) = 0$.
    Therefore, we obtain that $h_{\alpha} (t) \geq 0$ 
    for all $t \in \left[ 1 - \frac{1}{\alpha}, 1 \right)$.

    Next, consider the case of $t \geq 1$.
    Then we obtain $h'_{\alpha} (t) =  ((1 + \alpha) - \alpha t) / t$, 
    since this derivative is non-negative 
    for all $t \in \left[ 1, 1 + \frac{1}{\alpha} \right]$ 
    and $h_{\alpha} (1) = 0$, we have $h_{\alpha} (t) \geq 0$ 
    for all $t \in \left[ 1, 1 + \frac{1}{\alpha} \right]$.
\end{proof}
\begin{thm} \label{thm:pro_gua}
    Suppose that Assumption \ref{assum:pro_gua} holds.
    {We set}
    $a^0_i = \exp(\mu^{\ln}_i), \ i = 1, \dots, m$, and $A = (\Sigma^{\ln})^{-1/2}$,
    and {assume} $\ell_2$-norm is used for the norm term of the uncertainty set $\Omega(a^0, A, \tau)$. 
    If
    \begin{equation*}
        \tau \geq \delta(\varepsilon) \left( \exp \left(\lambda^{1/2} \delta(\varepsilon) \right) - 1 \right)
    \end{equation*}
    for $0 < \varepsilon < 1$, 
    then, for all robust feasible solutions $x(\tau)$ 
    to the constraint \eqref{ineq:rc_pro}, 
    the following inequality holds: 
    \begin{equation*}
        \prob{f(a,x(\tau)) \geq 0} \geq 1 - \varepsilon,
    \end{equation*}
    where 
    \begin{equation*}
        \delta(\varepsilon) = \sqrt{F^{-1}(m, 1 - \varepsilon)},
    \end{equation*}
    and $F(\nu, \cdot)$ is the distribution function of 
    the $\chi^2$-distribution with $\nu$ degrees of freedom.
\end{thm}
\begin{proof}
    For the proof, we define the sets 
    $\Omega_{\rm eq}(a^0, \tau, A)$ and $L(\delta)$ by 
    \begin{gather*}
        \Omega_{\rm eq}(a^0, \tau, A) = \left\{ 
            a \in \R^m \bigmid a = A^0 z, \ g(z) = y, \| Ay\|_2 \leq \tau
        \right\}, \\
        L(\delta) = \left\{ 
            a \in {\R^m _{++}} \bigmid \ln a = \mu^{\ln} + d, \ \left\| (\Sigma^{\ln})^{-1/2} d \right\|_2 \leq \delta
        \right\}.
    \end{gather*}   
    Since the robust feasible solutions $x(\tau)$ {satisfy}  
    $\min_{a \in \Omega} f(a, x(\tau)) \geq 0$, 
    we have 
    \begin{equation} \label{impl_feas}
        a \in \Omega(a^0, \tau, A) \implies 
        f(a, x(\tau)) \geq \min_{a \in \Omega(a^0, \tau, A)} f(a, x(\tau)) \geq 0.
    \end{equation}   
    Next, by the definition of $\Omega_{\rm eq}(a^0, \tau, A)$ it follows that 
    \begin{equation} \label{omega_omega_eq}
        \Omega_{\rm eq}(a^0, \tau, A) \subseteq \Omega(a^0, \tau, A).
    \end{equation}

    On the other hand, we show that $\Omega_{\rm eq}(a^0, \tau, A) \supseteq L(\delta(\varepsilon))$.
    Since $a_i$ are independent, $\Sigma$ is a diagonal matrix.
    Hence, $\Sigma^{\ln}$ and $A$ are also diagonal matrices 
    and the $i$th diagonal element of $A$ is $\lambda_i^{-1/2}$, 
    where $\lambda_i$ denotes the $i$th diagonal element of $\Sigma^{\ln}$.
    Therefore $L(\delta(\varepsilon))$ is expressed~as 
    \begin{align*}
        L(\delta(\varepsilon)) 
        &= \left\{ 
            a \in {\R^m _{++}} \bigmid \ln a = \mu^{\ln} + d, \ 
            \sum_{i = 1}^m \lambda_i^{-1} | d_i |^2 \leq \delta(\varepsilon)^2
            \right\} \\
        &= \left\{ 
            a \in {\R^m _{++}} \bigmid  
            \sum_{i = 1}^m \lambda_i^{-1} {\left(\ln \frac{a_i}{\exp(\mu^{\ln}_i)} \right)}^2 \leq \delta(\varepsilon)^2
            \right\} \\
        &= \left\{ 
            a \in {\R^m _{++}} \bigmid 
            \| A {y^L} \|_2 \leq \delta(\varepsilon)
            \right\},
    \end{align*}
    {where the last line follows by defining $y^L \in \R^m$ with $y^L_i \coloneqq \ln \frac{a_i}{a^0_i}$, $i = 1,\dots, m$.}
    and $\Omega_{\rm eq}(a^0, \tau, A)$ is rewritten as 
    \begin{equation*}
        \Omega_{\rm eq}(a^0, \tau, A) = \left\{ 
            a \in \R^m \bigmid g \left( \frac{a_i}{a^0_i} \right) = y_i \
            (i = 1, \dots, m), \ 
            \| Ay\|_2 \leq \tau
        \right\}.
    \end{equation*}
    Now let $a \in L(\delta(\varepsilon))$. Then $a$ satisfies 
    \begin{align}
        \left| \  \ln \frac{a_i}{a^0_i} \right| &= {y^L _i}, \quad i = 1, \dots, m, \label{ln_y} \\
        \| A {y^L } \|_2  &\leq \delta(\varepsilon). \label{L_Ay_delta}
    \end{align}
    From \eqref{L_Ay_delta} we have 
    $\lambda_i^{-1/2} {y^L _i} \leq \delta(\varepsilon), \ i = 1, \dots, m$, 
    which implies 
    \begin{equation*}
        0 \leq {y^L _i} \leq \lambda_i^{1/2} \delta(\varepsilon), \quad i = 1, \dots, m.
    \end{equation*}
    Then by \eqref{ln_y} and since $\lambda \geq \lambda_i > 0, \ i = 1, \dots, m$ we obtain 
    \begin{equation} \label{eval_arg_ln}
        \exp \left( - \lambda^{1/2} \delta(\varepsilon) \right) \leq 
            \frac{a_i}{a^0_i} \leq 
            \exp \left( \lambda^{1/2} \delta(\varepsilon) \right), \quad 
            i = 1, \dots, m.
    \end{equation}
    Let $z_i = a_i / a^0_i$ and 
    $\alpha = 1 / \left( \exp \left( \lambda^{1/2} \delta(\varepsilon) \right) - 1 \right)$ 
    in Lemma \ref{lemma:g_ln}.
    Then it follows that 
    \begin{equation*}
        \alpha g(z_i) \leq \left| \ln z_i \right|, \quad 
            \forall z_i \in \left[ 1 - \frac{1}{\alpha}, \  
            \exp \left( \lambda^{1/2} \delta(\varepsilon) \right) \right], \quad 
            i = 1, \dots, m.
    \end{equation*}
    Note that we can easily verify that 
    $\max \left\{ 0, 1 - \frac{1}{\alpha} \right\} \leq \exp \left(- \lambda^{1/2} \delta(\varepsilon) \right)$, 
    and then we have the above inequality for all 
    $t \in \left[ \exp \left( - \lambda^{1/2} \delta(\varepsilon) \right),  
    \exp \left( \lambda^{1/2} \delta(\varepsilon) \right) \right]$.
    Therefore, from \eqref{eval_arg_ln} it follows that 
    \begin{equation} \label{g_ln}
        a \in L(\delta(\varepsilon)) \implies \alpha g(z_i) \leq \left| \ln z_i \right|, \quad 
        i = 1, \dots, m,
    \end{equation}
    then, by letting $y^*_i = g(z_i), \ i = 1, \dots, m$, we have
    \begin{align*}
        \| A y^* \|_2^2 &= \sum_{i = 1}^m \lambda_i^{- 1} g(z_i)^2 \\
            &\leq \frac{1}{\alpha^2} \sum_{i = 1}^m \lambda_i^{- 1} (\ln z_i)^2 \\
            &= \frac{1}{\alpha^2} \| A \ln z \|_2^2 \\ 
            &\leq \frac{\delta(\varepsilon)^2}{\alpha^2} \\
            &= \left( \delta(\varepsilon) \left( \exp \left( \lambda^{1/2} \delta(\varepsilon) \right) - 1 \right) \right)^2 \\ 
            &\leq \tau^2,
    \end{align*}
    which implies $a \in \Omega_{\rm eq}(a^0, \tau, A)$.
    Thus, we obtain 
    \begin{equation} \label{l_omega_eq}
        L(\delta(\varepsilon)) \subseteq \Omega_{\rm eq} (a^0, \tau, A).
    \end{equation}

    Finally, since 
    $(\Sigma^{\ln})^{-1/2}(\ln a - \mu^{\ln})$
    follows the standard normal distribution, 
    \begin{equation*}
        \chi^2(a) = \left\| (\Sigma^{\ln})^{-1/2}(\ln a - \mu^{\ln}) \right\|_2^2
    \end{equation*}
    follows the $\chi^2$-distribution with $m$ degrees of freedom
    (see~\cite{Hog19}).
    Thus, if $\delta = \delta(\varepsilon)$, then we have
    \begin{equation} \label{p_logn}
        \prob{a \in L(\delta(\varepsilon))} 
            = \prob{\chi^2(a) \leq \delta(\varepsilon)^2}
            = 1 - \varepsilon,  
    \end{equation}
    which is justified 
    by the definition of the distribution function in~\cite{Hog19}. 

    Hence, it follows that 
    \begin{align*}
        \prob{f(a,x(\tau)) \geq 0} &\geq \prob{a \in \Omega(a^0, \tau, A)} \\
            & \geq \prob{a \in \Omega_{\rm eq}(a^0, \tau, A)} \\ 
            & \geq \prob{a \in L(\delta(\varepsilon))} \\
            & = 1 - \varepsilon. 
    \end{align*}
    The inequalities hold from 
    \eqref{impl_feas}, \eqref{omega_omega_eq}, and \eqref{l_omega_eq}, 
    respectively, 
    and the last equality is true from~\eqref{p_logn}. 
\end{proof}


\section{Numerical experiments} \label{sect:num}
In this section, we present numerical experiments 
in which the proposed model was applied to  
two different robust optimization problems driven by data.
Both problems have positive-valued parameters under uncertainty.
The objective of these experiments is 
to confirm that our model can assume parameter uncertainty 
only in $\R^m_{++}$. 

\subsection{Planning of the photovoltaic and battery system operation}
We first consider a simplified version of the problem given in~\cite{Mat14}, 
where we minimize the investment and operating costs of photovoltaic (PV) panels and batteries. The problem solved here is obtained 
by fixing the decision variables $z \in \R$ and $Q \in \R$ in the original one~\cite{Mat14}. 
Then the purpose of this problem is to optimize the operating schedule  
of the given PV panels and battery so as to minimize the operating cost.
The operating period is split into $T$ intervals $[t - 1, t]$, 
and the objective function is the sum of costs of purchased power at each interval.
The constraints ensure that the given power demand is satisfied at every interval, that the power generated by the PV does not exceed the solar irradiance, 
and that the power charged to the battery does not exceed its capacity.
The formulation is given as follows:
\begin{equation} 
     \label{prob:EO_E}
     \tag{$\text{EO}_{\text{E}}$}
     \begin{aligned} 
          & \hspace{-7mm} {\displaystyle \min_{x^C, x^P, x^R, x^S, q}}
               \quad \sum_{t = 1}^T C^P_t x^P_t \\
          & \ {\rm s.t.} \hspace{9mm} 
               x^R_t + \gamma x^S_t + x^P_t \geq D_t, \quad
               \forall t \in \mathcal{T}, \\
          & \hspace{14mm} 
               x^C_t + x^R_t \leq E_t z, \quad
               \forall t \in \mathcal{T}, \\
          & \hspace{14mm} 
               q{_t} = q{_{t - 1}} + x^C_t - x^S_t, \quad
               \forall t \in \mathcal{T}, \\
          & \hspace{14mm}
               q_0 = 0, \\
          & \hspace{14mm} 
               q_t \leq Q, \quad
               \forall t \in \mathcal{T}, \\
          & \hspace{14mm} 
               x^C, x^P, x^R, x^S, q \geq 0, \quad
               \forall t \in \mathcal{T},
     \end{aligned} 
\end{equation}
where $\mathcal{T} = \{1, \dots, 24\}$.
For simplicity, we omitted the constraint (5) in~\cite{Mat14} 
that restricts the ratio of purchased power to the whole supplied power.
Table~\ref{table:energy} shows the decision variables and parameters 
in the problem~\eqref{prob:EO_E}.

\begin{table}[ht]
     \caption{Definition of decision variables and parameters 
               in~\eqref{prob:EO_E}}
     \label{table:energy}
     \centering
     \small
     \begin{tabular}{@{}p{.03\textwidth}p{.40\textwidth}p{.03\textwidth}p{.40\textwidth}@{}}
          \toprule
          \multicolumn{2}{c}{Variables} & \multicolumn{2}{c}{Parameters} \\
          \cmidrule(r){1-2}\cmidrule(l){3-4}
          $x^C_t$ & power charged to the battery
          & $T$ & operation period [hour] \\
          $x^P_t$ & purchased power
          & $D_t$ & demand [kWh] \\
          $x^R_t$ & power supplied from PV
          & $E_t$ & solar irradiance [kWh/${\rm m}^2$] \\
          $x^S_t$ & power supplied from the battery
          & $z$ & size of PV panels [${\rm m}^2$] \\
          $q_t$ & charged power in battery
          & $Q$ & capacity of the battery [kWh] \\
          & & $C^P_t$ & electricity price [yen/kWh] \\
          & & $\gamma$ & discharge efficiency from the battery \\
          \bottomrule
     \end{tabular}
\end{table}

In this experiment, like in~\cite{Mat14}, we assume that 
the solar irradiance $E$ is an uncertain parameter, 
which can only have positive values.
The decision variables satisfy the second constraint in~\eqref{prob:EO_E} 
for all $E \in \Omega(E^0, \tau, A)$ if and only if 
\begin{align}
     & x^C_t + x^R_t \leq 
          (u^t)_t \ln \left( 1 + \frac{E^0_t z}{(u^t)_t} \right)
          - \tau \|s^t \|_*, \quad
          \forall t \in \mathcal{T}, \label{ineq:rc_log_energy} \\
     & u^t = A^{\T} s^t, \ u^t \geq 0, \quad
          \forall t \in \mathcal{T},  \notag \\
     & (u^t)_t + E^0_t z \geq 0, \quad
          \forall t \in \mathcal{T}, \label{ineq:rc_arg_log_energy}
\end{align}
where $E^0 \in \R^T$ denotes the nominal value of $E$, 
and $u^t, s^t \in \R^T, \ t \in \mathcal{T}$ are additional variables.

\begin{remark}
     This problem has the same form as the one in Example \ref{ex:infeas1}.
     Thus, if we use an existing uncertainty model, 
     then the robust counterpart becomes infeasible 
     when the size of the uncertainty set is large.   
\end{remark}

We set the value of parameters as follows.
The operation period is $T = 24$.
For hourly data of power demand $D$, we used 
the data in Kyushu on March 9, 2022, Japan from~\cite{Kyu}.
Note that we used the household demand data 
obtained by dividing the whole demand by the number of households from~\cite{Sou}.
As a nominal value $E^0$ of solar irradiance, 
we used the sample mean of the data at Fukuoka from March 1, 2022 to March 31, 2022 from~\cite{Met}.
We set the electricity price $C^P$ as
$C^P_t = 10.7$ if $t = 0, 1, 2, 3, 4, 5, 6, 7, 23$, 
otherwise $C^P_t = 32$.
Moreover, the discharge efficiency is set to $\gamma = 0.8$.

For the dual form \eqref{ineq:rc_log_energy} of the robust constraint, 
we used the $\ell_2$-norm.
To avoid numerical issues when the log argument becomes very small, we replaced the logarithmic function $\ln(\cdot)$ in \eqref{ineq:rc_log_energy} by its piecewise-linear lower approximation (tangent line) near $\varepsilon > 0$. Specifically, we used the term $\ell(x)$ instead of $\ln x$:
\begin{equation*}
     \ell(x) \coloneqq
     \begin{cases}
          \ln x, \quad x \geq \varepsilon, \\
          \frac{1}{\varepsilon} (x - \varepsilon) + \ln \varepsilon, \quad
          x < \varepsilon.
     \end{cases}
\end{equation*} 
Here we set $\varepsilon = 10^{-9}$.
Moreover, we can add the equal sign to \eqref{ineq:rc_arg_log_energy} 
because the robust optimal solution satisfies this constraint with inequality. 

The numerical experiments were conducted using MATLAB Online (Release R2025b). 
We solved the associated optimization problems utilizing the ``fmincon'' function from the MATLAB Optimization Toolbox.

Let us now analyze the results. Figure~\ref{fig:energy_optimal_config} shows the optimal operation plan of the nominal problem and its robust counterpart for $\tau=10 ^{-3.5}$. We can observe clear strategic differences between the two models. First, the robust optimal solution plans for significantly less direct solar generation due to the assumed worst-case uncertainty (the yellow PV line is visibly lower in Figure~\ref{fig:energy_optimal_config}(b) than in~\ref{fig:energy_optimal_config}(a)). Second, the battery charge level (the purple line $q$) remains near zero in the robust model. This happens because the robust model anticipates scarce solar resources and chooses to route all available PV power directly to immediate demand, rather than suffering the $20\%$ energy conversion loss ($\gamma=0.8$) associated with battery storage. Consequently, to maintain system safety without relying on the battery, the robust model plans to purchase significantly more power from the grid during daylight hours (the orange purchase line stays higher in Figure~\ref{fig:energy_optimal_config}(b) compared to the deep drop in~\ref{fig:energy_optimal_config}(a)). This demonstrates that the robust counterpart adapts its physical strategy to safely and efficiently handle severe weather uncertainty.

\begin{figure}[ht]
     \centering
     \begin{minipage}{.47\textwidth}
          \centering
          \includegraphics[width=.9\linewidth]
               {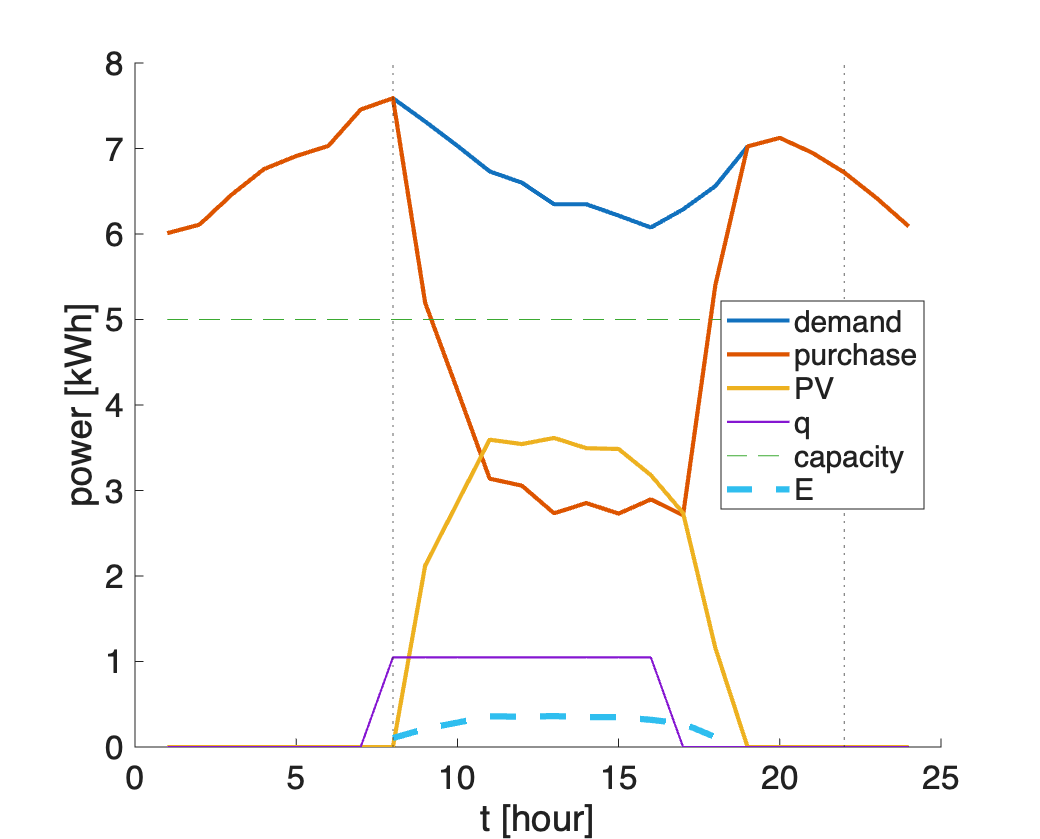}
          \subcaption{Nominal optimal solution}
     \end{minipage}
     \begin{minipage}{.47\textwidth}
          \centering
          \includegraphics[width=.9\linewidth]
               {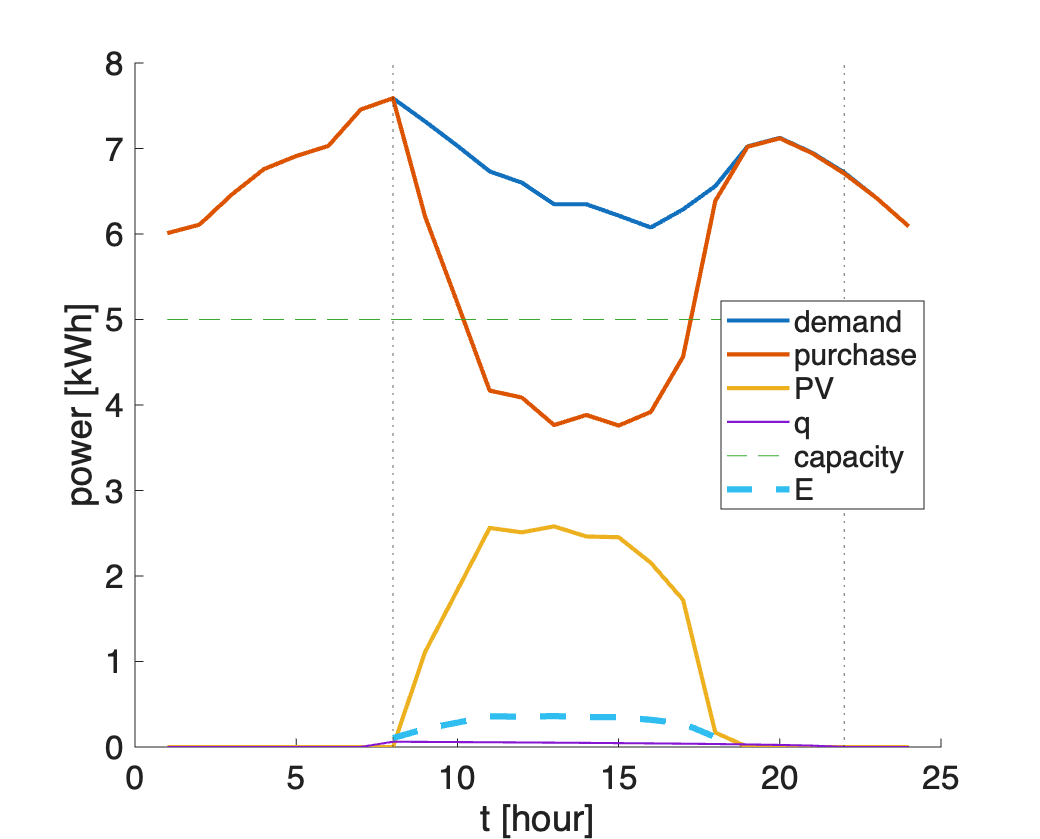}
          \subcaption{Robust optimal solution for \mbox{$\tau = 10^{-3.5}$}}
     \end{minipage}
     \caption{Optimal power supply configuration}
     \label{fig:energy_optimal_config}
\end{figure}

To further compare the optimal solutions of the robust counterpart 
to that of nominal problem~\eqref{prob:EO_E}, we use the following two criteria: maximum violation rate and actual cost. 
Throughout the description below, 
we denote the realized data of solar irradiance as $E^{\text{realized}}$.

The maximum violation rate is defined as follows:
\[
     v(E^{\text{realized}}, x^C, x^R) = 
     \max_{t \in \mathcal{T}} 
     \frac{x^C_t + x^R_t - E^{\text{realized}}_t z}
     {E^0_t z} \times 100.
\]
Note that $v(E^{\text{realized}}, x^C, x^R)$ measures 
how the solutions $x^C$ and $x^R$ violate the actual constraint
\begin{equation} \label{ineq:cons_realized}
     x^C_t + x^R_t \leq E^{\text{realized}}_t z, \quad
          \forall t \in \mathcal{T}, 
\end{equation} 
since $v(E^{\text{realized}}, x^C, x^R) \leq 0$ means that 
$x^C_t$ and $x^R_t$ satisfy the constraint \eqref{ineq:cons_realized} at the interval $t$
and this value is positive if the constraint is violated.

The actual cost is calculated by adjusting the optimal solutions. 
First, rescale $x^C_t$ and $x^R_t$ 
so that they satisfy the actual constraint \eqref{ineq:cons_realized} for all~$t$.
Second, rescale $x^S$ to make its sum over $t \in \mathcal{T}$ 
equal to the adjusted total value of charged power~$x^C$. 
Third, change the amount of purchased power~$x^P$ 
to fill the gap between both sides of the demand constraint 
\begin{equation*}
     x^R_t + \gamma x^S_t + x^P_t \geq D_t, \quad
               \forall t \in \mathcal{T},
\end{equation*}
caused by the reduction or increase of $x^R, x^S$.
Then, finally, the actual cost is given as the total purchase cost 
with respect to the adjusted values of~$x^P$.

Figure~\ref{fig:energy_criteria_tau_vals} shows the maximum violation rate and the relative difference of actual cost for valid values of $\tau$. As the realized values of solar irradiance, we used the data from March~1, 2023 to March~31, 2023~\cite{Met}. We can see from Figure~\ref{fig:energy_criteria_tau_vals}(a) that the nominal model results in a mean maximum violation rate of approximately 42.3$\%$. In contrast, the robust solutions successfully buffer the system against uncertainty, reducing the violation rate to below 40.5$\%$ and continuing to decrease it as the uncertainty set size $\tau$ increases. Furthermore, Figure~\ref{fig:energy_criteria_tau_vals}(b) demonstrates that the robust solutions achieve a consistent actual cost savings of over 1.1$\%$ compared to the nominal problem. Although there is a slight trade-off between these two criteria as~$\tau$ increases, the robust model strictly outperforms the nominal solution in both safety and cost.

\begin{figure}[ht]
     \centering
     \begin{minipage}{.45\textwidth}
          \centering
          \includegraphics[width=.9\linewidth]
               {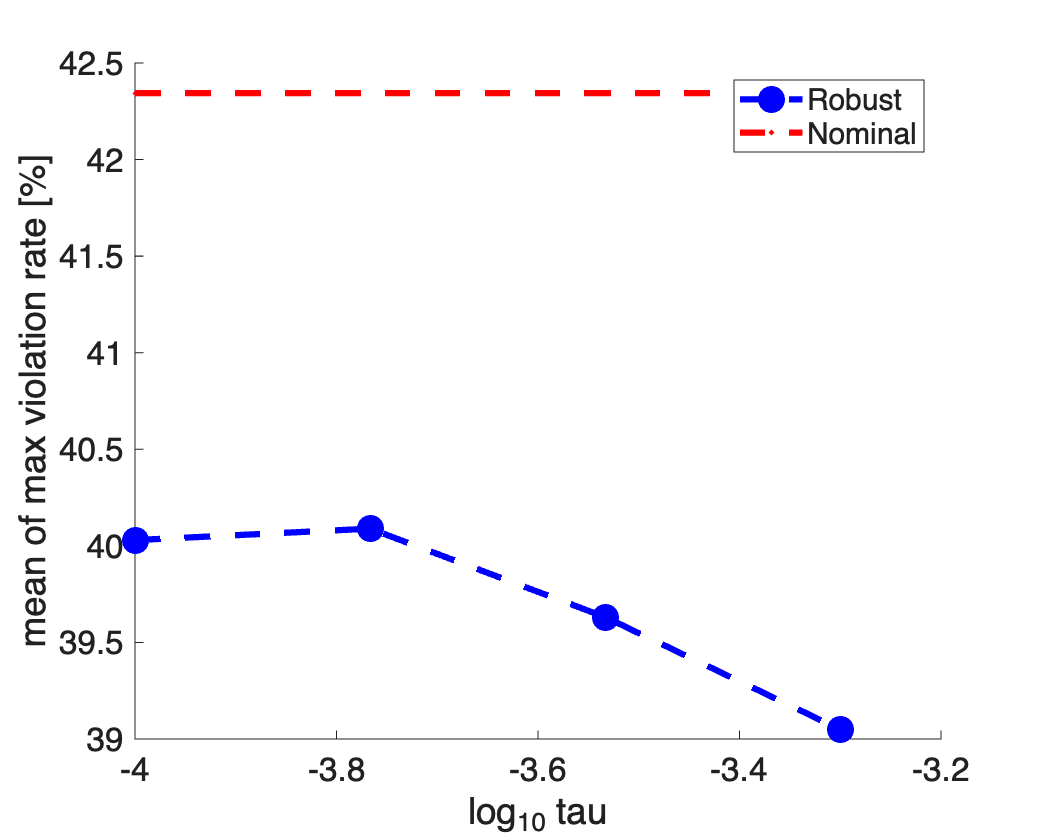}
          \subcaption{Maximum violation rate}
     \end{minipage}
     \begin{minipage}{.45\textwidth}
          \centering
          \includegraphics[width=.9\linewidth]
               {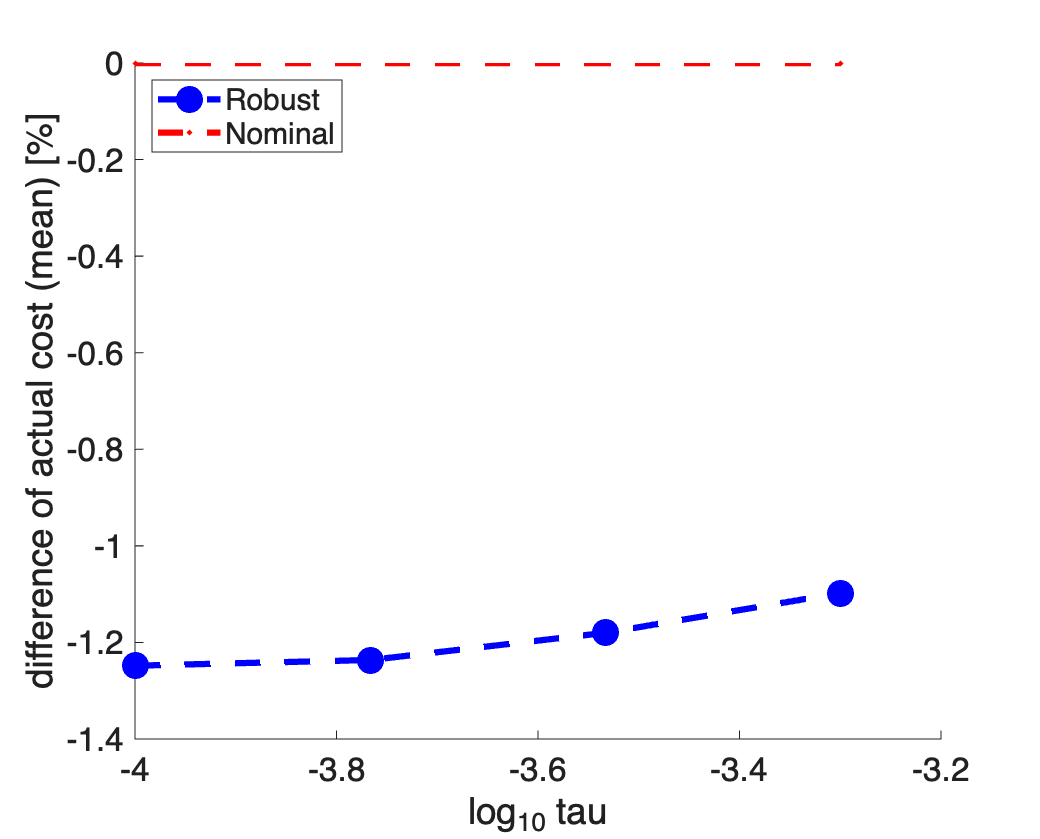}
          \subcaption{Relative difference of actual cost 
                         from that of the nominal problem}
     \end{minipage}
     \caption{Two criteria for some values of $\tau$}
     \label{fig:energy_criteria_tau_vals}
\end{figure}


\subsection{Binary classification}
Support vector machine (SVM) is a method that handles binary classification problems.
It finds a separating hyperplane 
by maximizing its margin against given training data.
In particular, C-SVM tolerates the misclassification to some degree, 
and enables us to obtain a separating hyperplane 
even if the data is not strictly separable. Its
formulation is given as follows:
\begin{equation} 
     \label{prob:svm}
     \tag{C-SVM}
     \begin{aligned} 
          & {\displaystyle \min_{w, b, \zeta}} \hspace{8mm}
            \frac{1}{2} \| w \|^2_2 + C \sum_{i = 1}^N \zeta_i \\
          & \ {\rm s.t.} \hspace{9mm} 
               y_i (w^{\T} x^i + b) \geq 1 - \zeta_i, \quad
               i = 1, \dots, N, \\
          & \hspace{14mm} 
               \zeta_i \geq 0, \quad
               i = 1, \dots, N, 
     \end{aligned} 
\end{equation}
where the separating hyperplane is expressed as $w^{\T} x + b = 0$, 
and $x^1, \dots, x^N \in \Rn$ denotes the given data.
Moreover, $y_i \in \{ -1, 1 \}$ is the class label of the $i$-th data.
The parameter $C > 0$ determines the balance of the regularization and the penalty term.

Existing research has shown that in machine learning problems such as SVM, 
``robustifying'' the model against uncertainties in the training data 
is equivalent to regularization.
In this experiment, we omitted the explicit regularization term from~\eqref{prob:svm}, 
and instead solved the robust counterpart 
obtained by assuming uncertainties in the training data using the proposed model. 
Let $\bar{x}^i \in \R^n_{++}$ denote the nominal feature vector (here $\bar{x}^i \coloneqq x^i$).
We assumed that the uncertain true data point $\tilde{x}^i$ belongs to the set $\Omega_i(\bar{x}^i, \tau, A^i)$ for $i = 1, \dots, N$, 
replacing the nominal constraint of~\eqref{prob:svm} with the robust constraint: 
\[
     \min_{\tilde{x}^i \in \Omega_i(\bar{x}^i, \tau, A^i)} y_i (w^{\T} \tilde{x}^i + b) \geq 1 - \zeta_i, \quad
     i = 1, \dots, N.
\]
Define $(A^0)^i \coloneqq \diag(\bar{x}^i)$. The equivalent dual form is given as follows:
\begin{equation*} 
     \begin{aligned} \label{problem:robust_svm}
          & {\displaystyle \min_{w, b, \zeta, \{u^i\}, \{s^i\} }} \hspace{8mm}
            \sum_{i = 1}^N \zeta_i \\
          & \ {\rm s.t.} \hspace{9mm} 
               \sum_{k = 1}^n u^i_k \ln 
               \left( 1 + \frac{y_i w_k \bar{x}^i_k}{u^i_k} \right)
               - \tau \| s^i \|_*
               + y_i b \geq 1 - \zeta_i, \quad
               i = 1, \dots, N, \\
          & \hspace{14mm}
               u^i = A^{\T} s^i, \ u^i \geq 0, \quad
               i = 1, \dots, N, \\ 
          & \hspace{14mm}
               u^i + y_i (A^0)^i w \geq 0, \quad
               i = 1, \dots, N, \\
          & \hspace{14mm} 
              \zeta_i \geq 0, \quad
               i = 1, \dots, N,
     \end{aligned} 
\end{equation*}
where $u^i, s^i \in \Rn$.

We used the Wine Quality data set from the UCI Machine Learning Repository~\cite{UCI} 
for the positive-valued data $x^i$.
We split 100 samples into 50 training data and 50 test data.
We implemented the experiment in the same environment as in Section 5.1. 

Let us now analyze the results.
Figure \ref{fig:svm_optimal_value} shows the change in the optimal value 
of the nominal~\eqref{prob:svm} and the robust counterpart 
when their parameters $C$ and $\tau$ were varied, respectively.
Figure \ref{fig:svm_accuracy} shows the corresponding change in accuracy against the test data 
of the nominal~\eqref{prob:svm} and the robust counterpart.
From both figures, we observe that 
the behavior of the optimal value and the accuracy of the robust solution as $\tau$ increases 
is remarkably similar to that of the nominal solution as $C$ decreases.
Note that in the nominal~\eqref{prob:svm} problem, a decrease in $C$ signifies that the regularization is strengthened.
From these results, we have confirmed that in the C-SVM framework, 
robustifying the constraints using our proposed model 
successfully yields a similar effect to explicit structural regularization.

\begin{figure}[ht]
     \centering
     \begin{minipage}{.45\textwidth}
          \centering
          \includegraphics[width=6cm]
          {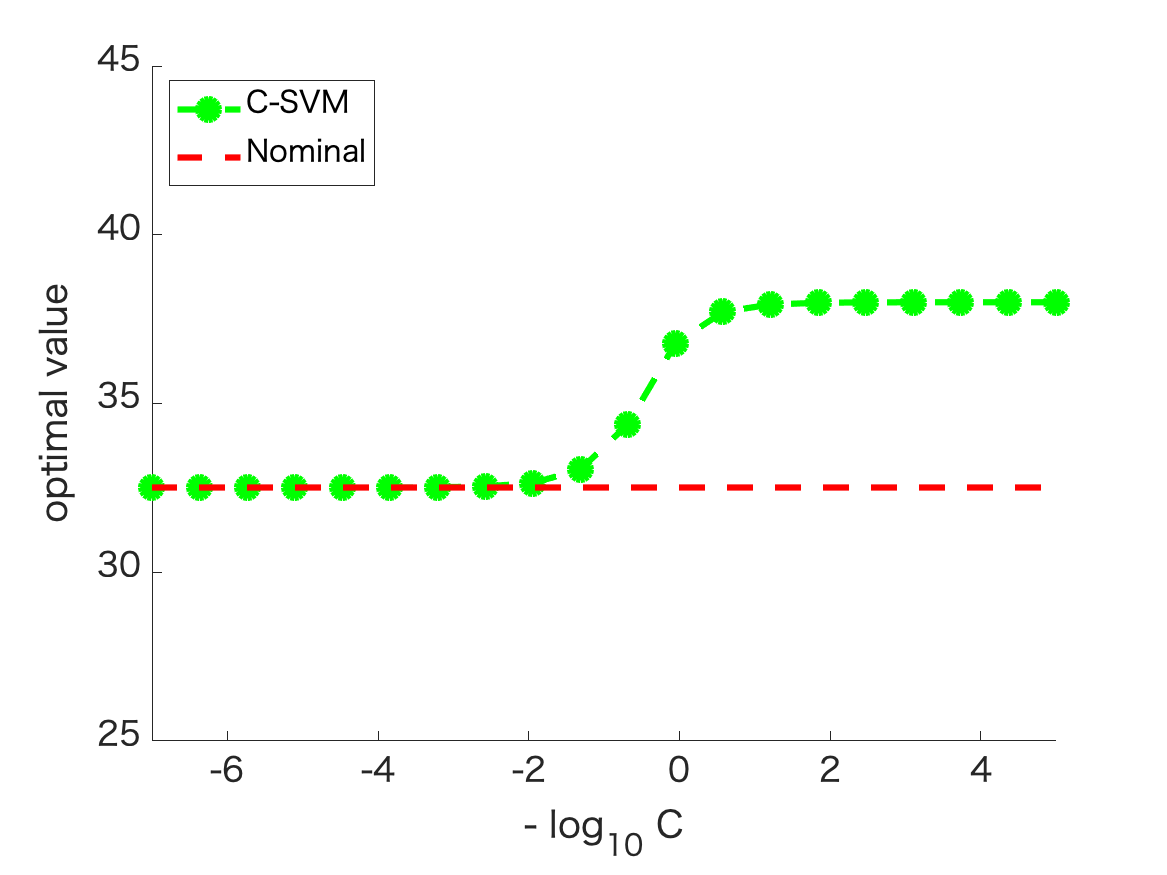}
          \subcaption{C-SVM}
          \label{fig:svm_opt_c-svm}
     \end{minipage}
     \begin{minipage}{.45\textwidth}
          \centering
          \includegraphics[width=6cm]
               {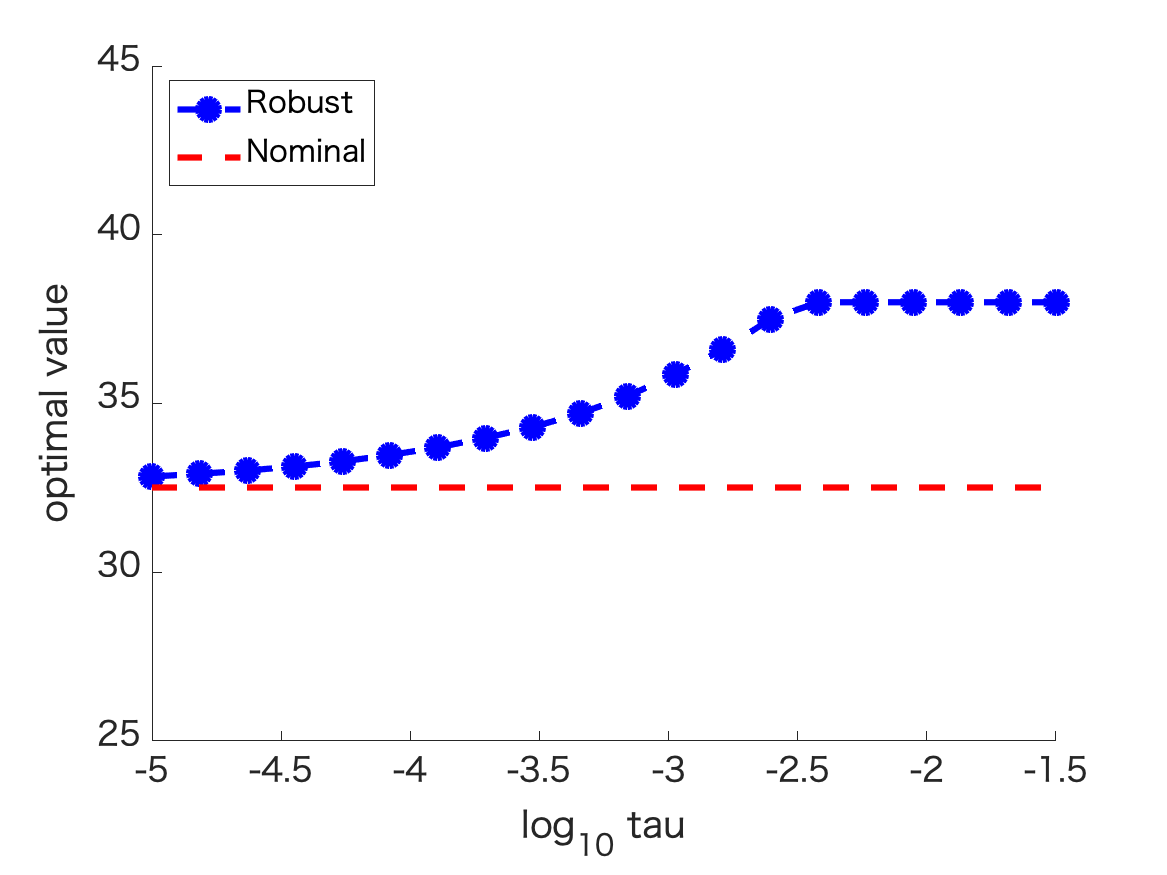}
          \subcaption{Robust counterpart}
          \label{fig:svm_opt_robust}
     \end{minipage}
     \caption{Optimal value}
     \label{fig:svm_optimal_value}
\end{figure}

\begin{figure}[ht]
     \centering
     \begin{minipage}{.45\textwidth}
          \centering
          \includegraphics[width=6cm]
          {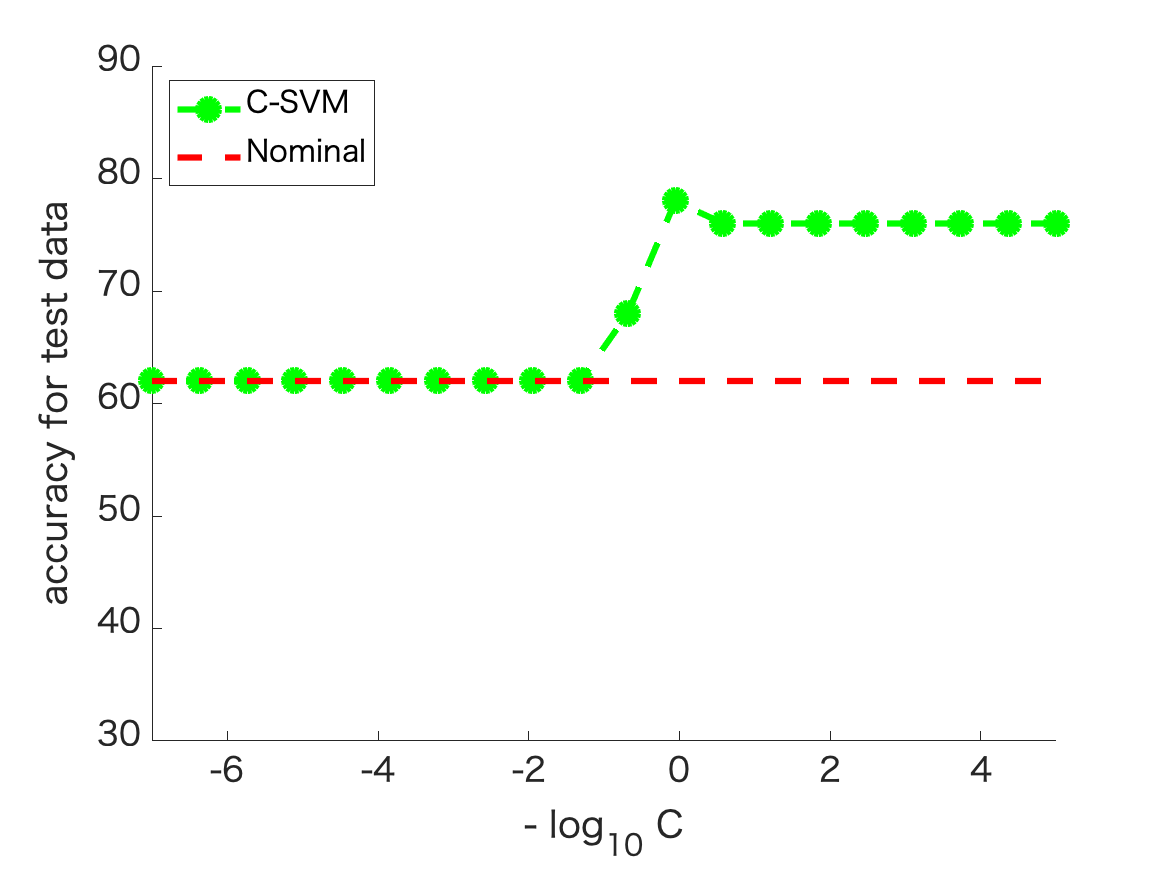}
          \subcaption{C-SVM}
          \label{fig:svm_acc_c-svm}
     \end{minipage}
     \begin{minipage}{.45\textwidth}
          \centering
          \includegraphics[width=6cm]
               {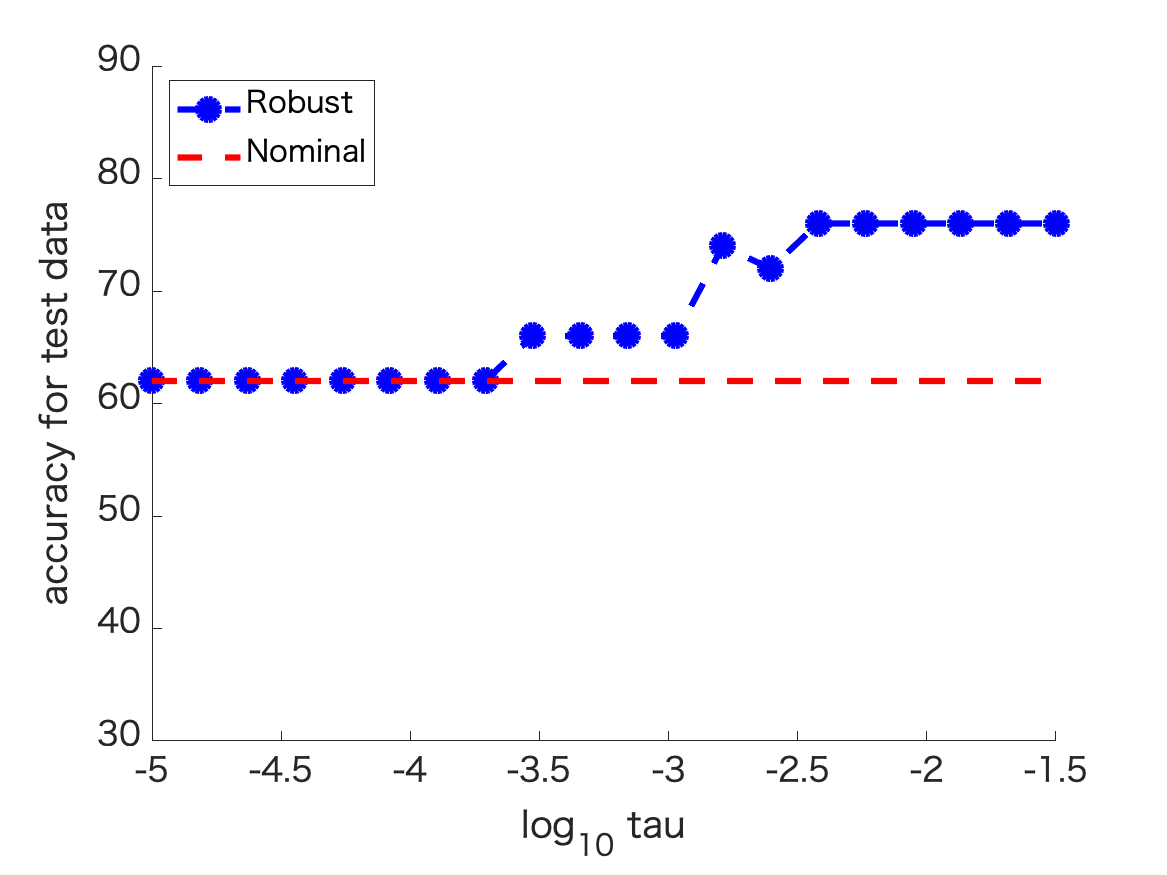}
          \subcaption{Robust counterpart}
          \label{fig:svm_acc_robust}
     \end{minipage}
     \caption{Accuracy for test data}
     \label{fig:svm_accuracy}
\end{figure}


\section{Conclusion} \label{sect:conc}

In this paper, we proposed a new uncertainty-set model designed for optimization problems with positive-valued parameters and derived a tractable dual reformulation of the corresponding robust counterpart. We established several properties of robust optimization problems under the proposed model, showing in particular that the model preserves positivity of the uncertain parameters and provides useful bounds that can guide the choice of the uncertainty level before solving the robust problem. We also derived a probabilistic guarantee for robust feasible solutions under a lognormal assumption on the uncertain parameters.

We also performed numerical experiments on two applications---PV--battery operation planning and a support vector machine model. For these problems, standard box/ellipsoidal uncertainty sets can lead to infeasibility at large uncertainty levels, whereas the proposed model yields a tractable robust counterpart. In our experiments, the robust solutions reduced constraint violations and also outperformed nominal solutions in terms of realized cost.
In addition to exploring other applications, an important direction for future work is to develop algorithms suited to the resulting robust counterpart, as it involves nonlinear constraints with logarithmic terms. It is also of interest to replace the function~$g$ used in the proposed model by a more general convex function.


\bibliographystyle{plain} 
\bibliography{references.bib}

\end{document}